\documentclass[sn-mathphys,Numbered]{sn_jnl}

\usepackage{easyReview}
\usepackage{enumerate}
\usepackage{graphicx}%
\usepackage{multirow}%
\usepackage{amsmath,amssymb,amsfonts}%
\usepackage{amsthm}%
\usepackage{mathrsfs}%
\usepackage{bm}
\usepackage[title]{appendix}%
\usepackage{xcolor}%
\usepackage{textcomp}%
\usepackage{manyfoot}%
\usepackage{booktabs}%
\usepackage{algorithm}%
\usepackage{algorithmicx}%
\usepackage{algpseudocode}%
\usepackage{listings}%
\usepackage{subfigure}
\usepackage{parskip}
\usepackage[section]{placeins}
\newcommand{\argmin}{\operatornamewithlimits{arg\,min}}

\theoremstyle{thmstyleone}%
\newtheorem{Theorem}{Theorem}
\theoremstyle{thmstyletwo}%

\theoremstyle{thmstylethree}%
\newtheorem{Assumption}{Assumption}%
\newtheorem{Lemma}{Lemma}%
\begin{document}

\title[Article Title]{Anderson acceleration for iteratively reweighted $\ell_1$ algorithm}


\author[1]{\fnm{Kexin} \sur{Li}}\email{likx1@shanghaitech.edu.cn}



\affil[1]{\orgdiv{School of Information Science and Technology}, \orgname{ShanghaiTech University}, \orgaddress{ \city{Shanghai}, \country{China}}}



\abstract{Iteratively reweighted L1 (IRL1) algorithm is a common algorithm for solving sparse optimization problems with nonconvex and nonsmooth regularization. The development of its acceleration algorithm, often employing Nesterov acceleration, has sparked significant interest. Nevertheless, the convergence and complexity analysis of these acceleration algorithms consistently poses substantial challenges. Recently, Anderson acceleration has gained prominence owing to its exceptional performance for speeding up fixed-point iteration, with numerous recent studies applying it to gradient-based algorithms. Motivated by the powerful impact of Anderson acceleration, we propose an Anderson-accelerated IRL1 algorithm and establish its local linear convergence rate. We extend this convergence result, typically observed in smooth settings, to a nonsmooth scenario. Importantly, our theoretical results do not depend on the Kurdyka-\L{}ojasiewicz condition, a necessary condition in existing Nesterov acceleration-based algorithms. Furthermore, to ensure global convergence, we introduce a globally convergent Anderson accelerated IRL1 algorithm by incorporating a classical nonmonotone line search condition. Experimental results indicate that our algorithm outperforms existing Nesterov acceleration-based algorithms.}



\maketitle
 
\newpage
\section{Introduction}\label{sec1}
In this paper, we aim to solve a class of nonconvex and nonsmooth optimization problems as follows:
\begin{equation}\label{equation1}\tag{P}
	\min\limits _{\bm{x}\in \mathbb{R}^n} \:F(\bm{x}):=f(\bm{x})+  \lambda \displaystyle\sum^n_{i=1} \phi (\vert x_i\vert),
\end{equation}
where $f: \mathbb{R} ^n \to \mathbb{R}$ is Lipschitz continuously differentiable function and $\phi:\mathbb{R}_+ \to \mathbb{R}_+$ is a nonconvex and nonsmooth regularization function. Additionally, a user-prescribed constant $\lambda \in \mathbb{R}_{++}$ often refers to the regularization parameter. Problem \eqref{equation1} arises from diverse fields ranging from signal processing \cite{philiastides2006temporal,tropp2006just}, biomedical informatics \cite{tsuruoka2007learning,liao2007logistic} to modern machine learning \cite{mairal2010online,scardapane2017group}.


The function $\phi$ assumes various forms such as EXP approximation \cite{bradley1998feature}, LPN approximation \cite{fazel2003log}, LOG approximation \cite{lobo2007portfolio}, FRA approximation \cite{fazel2003log}, TAN approximation \cite{candes2008enhancing}, SCAD \cite{fan2001variable}, and MCP approximation \cite{zhang2007penalized}. In Table \ref{tab1}, we delineate the explicit forms of different cases. Here, $p$ is a hyperparameter that governs the sparsity. These common $\ell_0$ approximation functions adhere to the assumptions outlined below:

    



        




    


\begin{table}[h]
    \renewcommand\arraystretch{1.5}
    \caption{Different weights based on different choice of $\phi$}\label{tab1}
    \begin{tabular}{@{}lll@{}}

    \toprule
    $\phi$ & $\phi(|\tilde{x}_i|)$  & $\tilde{\omega}_i$  \\
    \midrule
    
    EXP \cite{bradley1998feature}    & $1-e^{-p|\tilde{x}_i|}$   & $ pe^{-p(\|\tilde{x}_i\|_1)}$  \\

    LPN \cite{fazel2003log}   & $\tilde{x}^p_i$   & $p(\|\tilde{x}_i\|_1+\tilde{\epsilon}_i)^{p-1}$ \\
        
    LOG \cite{lobo2007portfolio}    & $\log(1+p|\tilde{x}_i|)$   & $\frac{p}{1+p\|\tilde{x}_i\|_1} $ \\

    FRA \cite{fazel2003log}   & $\frac{|\tilde{x}_i|}{|\tilde{x}_i|+p}$   & $\frac{p}{(\|\tilde{x}_i\|_1+p)^2} $  \\

    TAN \cite{candes2008enhancing}    & $\frac{|\tilde{x}_i|}{|\tilde{x}_i|+p}$   & $\frac{p}{1+p^2(\|\tilde{x}_i\|_1)^2} $ \\
    \botrule
    \end{tabular}

    \end{table}
 
    Despite the potential wide applicability of \eqref{equation1}, its nonconvex and nonsmooth nature generally makes it challenging to solve. Numerous researchers have focused on the design and theoretical analysis of algorithms for such nonconvex and nonsmooth sparse optimization problems \cite{wang2021nonconvex,gong2013general,lu2014iterative,wang2015optimality,yang2022towards}. The iteratively reweighted $\ell_1$ algorithm, due to its powerful performance, has obtained extensive attention \cite{lu2014iterative,wang2022extrapolated,yu2019iteratively,chartrand2008iteratively,lu2014proximal,chen2010convergence}. IRL1 is a special case of the Majorization-minimization algorithm. The core of IRL1 approximates the nonsmooth regularization by a weighted $\ell_1$ norm in each iteration. Therefore, IRL1 iteratively solves a local convex approximation model to address the objective problem \eqref{equation1}. 
    
    Since Nesterov introduced the extrapolation technique for the gradient descent algorithm \cite{nesterov1983method}, a variety of momentum-based acceleration techniques have been utilized in first-order algorithms. A notable example is the Fast Iterative Shrinkage-Thresholding Algorithm (FISTA) proposed by Beck and Teboulle \cite{beck2009fast}. Concurrently, numerous studies have been dedicated to the design of acceleration algorithms for the IRL1 algorithm. However, due to the non-convexity and non-smoothness of problem \eqref{equation1}, analyzing the convergence and complexity of the acceleration IRL1  algorithm often poses significant challenges. Existing literature primarily utilizes Nesterov's extrapolation technique to accelerate the algorithm. These Nesterov acceleration-based algorithms typically rely on the Kurdyka-\L{}ojasiewicz (KL) condition to establish convergence and complexity, with their complexity outcomes frequently tied to the KL coefficient. However, numerous practical problems cannot fully ensure the KL property. Furthermore, estimating the KL coefficient for a given function often presents substantial challenges.  
    
    The current main results of the accelerated iteratively reweighted $\ell_1$ algorithm are as follows: Yu and Pong \cite{yu2019iteratively} proposed three different versions of the IRL1 acceleration algorithms with the Nesterov technique and demonstrated the global convergence of the algorithms under the Kurdyka-\L{}ojasiewicz (KL) conditions. However, they did not conduct a complexity analysis. Additionally, they stipulated that $f$ must be a smooth convex function and that the limit $\lim \limits_{t \rightarrow 0^+} \phi(t)$ must exist. This requirement restricts the use of certain regularization, such as the commonly used LPN regularization. Furthermore, Wang and Zeng \cite{wang2022extrapolated} introduced an accelerated iteratively reweighted L1 algorithm with the Nesterov technique, specifically for LPN regularization. They provided a global convergence guarantee and local complexity analysis for their algorithm under the Kurdyka-\L{}ojasiewicz (KL) conditions. The algorithm has local linear convergence when the KL coefficient is not greater than $1/2$, and local sublinear convergence when the KL coefficient exceeds $1/2$.
    


From another perspective, classical sequence-based extrapolation techniques have garnered increasing attention. These include minimal polynomial extrapolation \cite{smith1987extrapolation}, reduced rank extrapolation \cite{eddy1979extrapolating}, and Anderson acceleration \cite{anderson1965iterative}. Initially, these techniques are commonly used method for accelerating fixed-point iteration. Anderson acceleration has obtained great attention due to its powerful performance, which leverages the information from historical iterates and sums the weighted historical iterates to obtain new iterate. Anderson acceleration was initially introduced to accelerate the solution process of nonlinear integral equations \cite{anderson1965iterative}. This technique was later generalized to address the general fixed-point iteration \cite{walker2011anderson,toth2015convergence}. Currently, Anderson acceleration has been widely used and analyzed. In recent times, numerous works have integrated Anderson acceleration into optimization algorithms. Fu, Zhang, and Boyd \cite{fu2020anderson} applied it to the Douglas-Rachford Splitting algorithm. Poon and Liang \cite{poon2019trajectory} utilized Anderson acceleration to augment the performance of the ADMM method.  The design and convergence analysis of Anderson accelerated algorithm has been the subject of comprehensive study. Although Anderson acceleration does not use derivatives in its iterations, proving its convergence without continuous differentiability presents a challenge. Convergence analyses have been conducted for the linear case \cite{toth2015convergence,walker2011anderson}, the continuously differentiable case \cite{chen2019convergence}, and the Lipschitz continuously differentiable case \cite{toth2017local,toth2015convergence}. Recent studies have presented local convergence and local linear complexity outcomes for Anderson acceleration in certain specific non-smooth fixed point iteration problems. For a fixed point issue that can be divided into the sum of a smooth component and a non-smooth component with a minor Lipschitz constant, Bian et al.\cite{bian2021anderson} have demonstrated new findings. Furthermore, Bian and Chen\cite{bian2022anderson} have offered convergence and complexity outcomes for a non-smooth composite fixed-point iteration problem. Additionally, Mai and Johansson  \cite{mai2020anderson} applied Anderson acceleration to the classic proximal gradient method.

\subsection{Main contributions}

Considering the success of Anderson acceleration for accelerating various problems, we proposed Anderson accelerated iteratively reweighted $\ell_1$ algorithm. The main contributions of this paper can be summarized as follows:

\begin{enumerate} [1)]

    \item We present an Anderson accelerated iteratively reweighted $\ell_1$ algorithm for solving common nonconvex and nonsmooth regularization optimization problem as defined in \eqref{equation1}. Under the framework of Anderson acceleration, we establish the local R-linear convergence rate of the algorithm in the nonsmooth setting. Notably, such results are typically observed only in a smooth setting. Moreover, to the best of our knowledge, this is the first work that guarantees linear complexity results for accelerated IRL1 algorithms without the need for KL conditions.
    
    \item Considering the existing uncertainty about the global convergence of the Anderson acceleration, we introduce a safety strategy rooted in a classical nonmonotone line search condition in \cite{zhang2004nonmonotone}. We present a globally convergent Anderson-accelerated IRL1 algorithm based on the strategy. Our algorithm ensures global convergence while preserving computational efficiency.
    
    \item In our experiments, we verified the theoretical results of our algorithms and demonstrated that they outperform the state-of-the-art accelerated IRL1 algorithm proposed in \cite{wang2022extrapolated}.
    \end{enumerate}

\subsection{Notation and preliminaries}\label{subsec11}

We denote $\mathbb{N}:=\{0,1,2...\} $, $\mathbb{R}^n$ as the n-Dimensional Euclidean space and $\mathbb{R}_+^n$ as the positive orthant in $\mathbb{R}^n$, with $\mathbb{R}_{++} ^n$ being the interior of $\mathbb{R}_+^n$. Let $\|\bm{x}\|_p = (\sum^{n}_{i=1}|x_i|^p)^{1/p}$ with $p\in (0,+\infty)$ and $\|\bm{x}\|$ as the $\ell_2$ norm. Define ${\rm sign}(\bm{x})=({\rm sign}(x_1),...,{\rm sign}(x_n)^T$.  Define ${\rm diag}(\bm{x})\in \mathbb{R}^{n\times n}$ be  diagonal matrix with
$\bm{x}$ forming the diagonal elements. Define $\mathcal{A} (\bm{x}):=\{\bm{x}\in\mathbb{R}^n|i:x_i=0\}$ and $\mathcal{I} (\bm{x}):=\{\bm{x}\in\mathbb{R}^n|i:x_i\neq 0\}$. Consider a metric space (M,d), if a fixed-point mapping $H:\mathbb{R}^n \rightarrow \mathbb{R}^n$ exists such that for any real number $\gamma \in (0,1)$ and for all $\bm{x},\bm{y}$ in $M$, the inequality $\|H(\bm{x})-H(\bm{y})\|\le \gamma \|\bm{x}-\bm{y}\|$ holds, then we define $H$ as a contraction mapping from $M$ to itself. For a set $S$, we define the relative interior ${\rm rint}(S)$ be:

$${\rm rint}(S):=\{\bm{x}\in S:{\rm there\: exists}\: \epsilon >0\: {\rm such\: that}\: B_{\epsilon}(\bm{x})\cap {\rm aff}(S)\subseteq S \},$$
where ${\rm aff}(S)$ is the affine hull of $S$ and $B_{\epsilon}(\bm{x})$ is a open ball of radius $\epsilon$ centered on $\bm{x}$. The subdifferential of a convex function $f:\mathbb{R}^n \rightarrow \mathbb{R}$ at $\bm{x}$ is the set defined by 
$$ \partial f(\bm{x}) = \left\{\bm{z}\in \mathbb{R} ^n  : f(\bm{y})-f(\bm{x})\ge\langle \bm{z},\bm{y}-\bm{x} \rangle , \forall \bm{y} \in \mathbb{R} ^n   \right\}.$$
 For an extended-real-valued function $f:\mathbb{R}^n \rightarrow (-\infty,+\infty]$, if the domain $\mathop{{\rm dom}}f:=\{\bm{x}:f(\bm{x})<+\infty\} \ne \emptyset  $, we say $f$ is proper. A proper function $f$ is said to be closed if it is lower semi-continuous. For a proper function $f$, its Fréchet subdifferential at $\bm{x}$, denoted as $\partial_F f(\bm{x})$, is the set

$$ \partial_F f(\bm{x}) = \left\{\bm{u} \in \mathbb{R} ^n  : \mathop{{\rm lim \, inf}} _{\bm{z}\rightarrow \bm{x} ,\bm{z}\ne \bm{x}} \frac{f(\bm{z})-f(\bm{x})-\langle \bm{u},\bm{z}-\bm{x} \rangle }{\|\bm{z}-\bm{x}\|}\ge 0 \right\}.$$
Let the distance of a point $\bm{x}\in\mathbb{R}^n$ to a set $S\subset \mathbb{R}^n$ be 
$${\rm dist}(\bm{x},S):={\rm inf}\{\|\bm{x}-\bm{s}\|:\bm{s}\in S\}.$$
The following assumptions are supposed
to be satisfied by the minimizing function  $F$ in \eqref{equation1} throughout the paper.
\begin{Assumption}\label{Assumption1} 
    \begin{itemize}
        \item [(i)] The function $f$ has Lipschitz continuous gradient with constant $L_f > 0$, that is, $f(\bm{x})-f(\bm{y}) \le \nabla f(\bm{y})^T(\bm{x}-\bm{y})+\frac{L_f}{2}\|\bm{x}-\bm{y}\|^2_2,\ \forall\ \bm{x},\bm{y}\in\mathbb{R}^n$.
            
        \item [(ii)] The function $\phi$ is smooth, concave and strictly increasing on $(0, +\infty)$. Moreover, $\phi(0) = 0$ and it is Fr\'echet subdifferentiable at $0$.
            
        \item [(iii)] The function $F$ is non-degenerate in the sense that for any critical point $\bar{\bm{x}}$ of $F$, it holds that $\bm{0} \in \mathop{{\rm rint}}\partial_F F(\bar{\bm{x}})$.
    \end{itemize}
\end{Assumption}


 Next, we recall some basic background of the iteratively reweighted $\ell_1$ algorithm and Anderson acceleration methods.

\textbf{The basis of iteratively reweighted $\ell_1$ (IRL1) algorithm}.The basis of iteratively reweighted $\ell_1$ (IRL1) algorithm The basic version IRL1 is a special case of majorization-minimization instance. To overcome the nonsmooth objective in \eqref{equation1}, a common technique used in the IRL1 algorithm is to add a perturbation vector $\bm{\epsilon}\in \mathbb{R}^n_{++}$ to the nonsmooth term $\phi$ to have a continuously differentiable relaxed objective $F(\bm{x},\bm{\epsilon})$, i.e.,
\begin{equation}
    \begin{aligned}\label{Eq_relaxedModel}
        F(\bm{x},\bm{\epsilon}):= f(\bm{x})+ \lambda \sum^n_{i=1}\phi(|x_i|+\epsilon_i),
    \end{aligned}
\end{equation}
With such a relaxed optimization model, a key to using IRL1 to solve \eqref{Eq_relaxedModel} is to derive a convex majorant. Specifically, given the $k$th iterate, we have
\begin{equation}\label{regularization approximation}
    \begin{aligned}
        F(\bm{x},\bm{\epsilon}) &= f(\bm{x})+ \lambda \sum^n_{i=1}\phi(|x_i|+\epsilon_i)\\
                                &\leq f(\bm{x}^{k}) + \nabla f(\bm{x}^{k})^{T}(\bm{x} - \bm{x}^{k}) + \frac{L}{2}\vert\bm{x}-\bm{x}^{k}\Vert_2^2\\
                                &\quad + \lambda \sum_{i=1}^{n}\phi'(\vert x_i^{k} \vert + \epsilon_i^{k})(\vert x_i \vert - \vert x_i^{k}\vert),\\
    \end{aligned}
\end{equation}
where the first inequality holds mainly because of the $L_f$-smoothness of $f$ and the concavity of $\phi$, and $L \ge L_f$. Consequently, at the $k$-th iteration, one solves a convex subproblem that locally approximates $F$ to obtain the new iterate $\bm{x}^{k+1}$, i.e., 
\begin{equation}
	\begin{aligned}\label{equation3}
		\bm{x}^{k+1} \leftarrow  \argmin_{x\in\mathbb{R}^n}\{G(\bm{x};\bm{x}^k,\bm{\epsilon}^k):= Q_k(\bm{x})+\lambda \sum^n_{i=1}\omega^k_i\vert x_i\vert\},
\end{aligned}
\end{equation}
where $\omega^k_i:=\omega(x^k_i,\epsilon^k_i) =\phi'(|x_i^k|+\epsilon^k_i)$ and  $Q_k(\bm{x}) = Q(\bm{x};\bm{x}^k):=  \nabla f(\bm{x}^k)^T \bm{x}+\frac{L}{2}\|\bm{x}-\bm{x}^k\|^2$.

The subproblem (\ref{equation3}) has a closed-form solution:

\begin{equation}\label{equation6}
	x^{k+1}_i =
	\begin{cases}
		x^k_i -\frac{1}{L}[\nabla_i f(\bm{x}^k)-\lambda \omega^k_i]   & x^k_i<\frac{1}{L}[\nabla_i f(\bm{x}^k)-\lambda \omega^k_i] ,\\
        x^k_i -\frac{1}{L}[\nabla_i f(\bm{x}^k)+\lambda \omega^k_i]   & x^k_i>\frac{1}{L}[\nabla_i f(\bm{x}^k)+\lambda \omega^k_i] ,\\
		0   & otherwise.
	\end{cases}
\end{equation}

In the IRL1 algorithm, the value of $\bm{\epsilon}$ has a significant impact on the performance. A large $\epsilon$ makes the subproblem have good properties but will cause the algorithm to miss many local minimum solutions, thereby affecting hitch the performance of the algorithm. Conversely, a small $\epsilon$ fails to smooth the problem, leading to challenges in solving the subproblem. Hence, a common strategy is to initialize the algorithm with a large $\bm{\epsilon}^0$  and gradually reduce it to zero over the iterations. 
 Wang et al. \cite{wang2021relating} and  Lu \cite{lu2014iterative} have each proposed dynamic update methods for $\bm{\epsilon}$. Additionally, Wang et al. \cite{wang2021relating} points the locally stable sign property of the $\{\bm{x}^k\}$ that is generated by IRL1 with $\ell_p$ norm, which means that $\textrm{sign}(\bm{x}^k)$ remain unchanged for sufficiently large $k$.

Denote $W^k:=\mathop{{\rm diag}}(\omega^k_1,...\omega^k_n)$. The ﬁrst-order necessary optimality condition of the subproblem (\ref{equation3}) is given as follows

\begin{equation}
    \begin{aligned} \label{equation11}
        \nabla Q_k(\bm{x}^{k+1}) + \lambda W^k \xi^k = 0,
\end{aligned}
\end{equation}
where $\xi^k \in \partial \|\bm{x}^{k+1}\|_1$. 

For $F(\bm{x},\bm{\epsilon})$ and the local model \(Q_k(\cdot)\), we make the following assumption:
    
    \begin{Assumption}\label{Assumption2} For $F(\bm{x}^0,\bm{\epsilon}^0)$ and $Q_k(\cdot)$, the following statements hold.
            \begin{itemize}		
                \item [(i)] The level set $\mathcal{L}(F(\bm{x}^0,\bm{\epsilon}^0)) :=\{x | F(x) \leq  F(\bm{x}^0,\bm{\epsilon}^0) \}  $ is bounded.
                \item [(ii)]
                For all $k \in \mathbb{N}$, $Q_k(\cdot)$ is strongly convex with constant $M>L_f/2>0$ and Lipschitz differentiable with constant $L>0.$
            \end{itemize}
        \end{Assumption}
     
These assumptions for local model $Q_k(\cdot)$ and $\mathcal{L}(F(\bm{x}^0,\bm{\epsilon}^0))$ are common in much of the literature on IRL1. They ensure the descent of the function value along the sequence of iterates $\{\bm{x}^k\}$ and also guarantee the boundness of $\{\bm{x}^k\}$.

\textbf{Anderson acceleration}. Anderson acceleration initially is used for solving such a fixed-point iteration:
\begin{equation}\label{fixed-point iteration}
     \begin{aligned}
         \mbox{Find } \bm{x} \in \mathbb{R}^n \mbox{ such that } \bm{x} = H(\bm{x}),
 \end{aligned}
 \end{equation}
where $H:\mathbb{R}^n\rightarrow \mathbb{R}^n$ is a mapping.  The framework of Anderson accelerated fixed-point iterations is given as follows:

 \begin{algorithm}[H]
     \renewcommand{\algorithmicrequire}{\textbf{Input:}}
     \renewcommand{\algorithmicensure}{\textbf{Initialization:}}
     \caption{Anderson accelerated fixed-point iteration}\label{AA fixed-point}
     \begin{algorithmic}[1]
     \Require Given $ \bm{x}^0\in\mathbb{R}^n$. Pick $m\ge 1$.
     \Ensure Set $k = 0 $ 
     \While{not convergence}
     \State Set $m_k =\min (m,k)$
     \State Set  $R^k = [\bm{r}^k,...,\bm{r}^{k-m_k}]$, where $\bm{r}^k = H(\bm{x}^k) - \bm{x}^k$
     
     \State Update $\bm{\alpha}^k \leftarrow \argmin_{\bm{\alpha}^T\textbf{e}=1}\|R^k\bm{\alpha}\|$
     
     \State Update $\bm{x}^{k+1} = \sum^{m_k}_{i=0}\alpha^k_i H^{k-m_k+i}$
     
     \State Set $k=k+1$
     \EndWhile
     \end{algorithmic}
     \end{algorithm}
We denote $H^k=H(\bm{x}^k)$ as the vector obtained by mapping $\bm{x}^k$ through $H$. Define $\bm{r}^k=H(\bm{x}^k)-\bm{x}^k$ as the residual term and $R^k  = [\bm{r}^k,...,\bm{r}^{k-m_k}]$, where $R^k \in \mathbb{R}^n \times \mathbb{R}^{m+1}$. Given an initial $x^0$ and an integer parameter $m\ge 1$, the fundamental concept of Anderson acceleration is to derive a weight vector $\bm{\alpha} \in \mathbb{R}^{m+1}$ that minimizes the weighted sum of the previous $m+1$ residual terms $r^k$:
\begin{equation}
    \begin{aligned}\label{AA subproblem}
        \bm{\alpha}^k \leftarrow \argmin_{\bm{\alpha}^T\textbf{e}=1}\|R^k\bm{\alpha}\|.
\end{aligned}
\end{equation}
Subsequently, the $\bm{x}^{k+1}$ is obtained by computing the weighted sum of $H^{k},...,H^{k-m_k}$, using the weight vector $\bm{\alpha}^k$.

\begin{equation}
    \begin{aligned}
        \bm{x}^{k+1} = \sum^{m_k}_{i=0}\alpha^k_i H^{k-m_k+i}.
\end{aligned}
\end{equation}
The subproblem (\ref{AA subproblem}) has a closed-form solution

\begin{equation}
    \begin{aligned}
        \bm{\alpha}^{k} = \frac{[(R^k)^TR^k]^{-1}\boldsymbol{1}}{\boldsymbol{1}^T [(R^k)^TR^k]^{-1}\boldsymbol{1}}.
\end{aligned}
\end{equation}
The cost of solving (\ref{AA subproblem}) is $O(m^2+mn)$. Given that $m$ is typically a very small (common setting is 5 or 15) constant in practice, the computation involved in solving this subproblem is considered trivial. Therefore, Anderson acceleration does not excessively increase the computation. It is worth noting that $(R^k)^TR^k$ may be singular. To guarantee the non-singularity of the least-square subproblem, we can add a Tikhonov regularization of $10^{-10}\|R^k\|^2$ to it (similar to \cite{mai2020anderson, scieur2017nonlinear}).

\section{Anderson accelerated IRL1}\label{sec2}

In this section, we describe the proposed algorithm with Anderson acceleration for solving \eqref{equation1}.

In our algorithm, we update $\bm{\epsilon}$ by $\bm{\epsilon}^{k+1}=\mu \bm{\epsilon}^{k}$ in iteratively reweighted $\ell_1$ algorithm, where $\mu\in (0,1)$ is a coefficient that controls the decay rate. Then the iteration of the IRL1 algorithm can be regarded as a fixed-point iteration $H:\mathbb{R} ^{2n}\rightarrow \mathbb{R}^{2n}$ of the compound variable $(\bm{x},\bm{\epsilon})$.

\begin{equation}
	\begin{aligned}\label{H}
        \left[\begin{array}{c}
            \bm{x}^{k+1}  \\
            \bm{\epsilon}^{k+1}
            \end{array}\right]
             = H(\bm{x}^k,\bm{\epsilon}^k) := 
        \left[\begin{array}{c}
            H_{\bm{x}}(\bm{x}^k,\bm{\epsilon}^k)  \\
            H_{\bm{\epsilon}}(\bm{x}^k,\bm{\epsilon}^k)
            \end{array}\right]
            = 
        \left[\begin{array}{c}
            \argmin_{\bm{x} \in\mathbb{R}^n} G_k(\bm{x})  \\
            \mu\bm{\bm{\epsilon}}^{k} 
            \end{array}\right].
\end{aligned}
\end{equation}

Where $H_{\bm{x}}(\bm{x}^k,\bm{\epsilon}^k) = \argmin_{\bm{x} \in\mathbb{R}^n} G_k(\bm{x})$ and $H_{\bm{\epsilon}}(\bm{x}^k,\bm{\epsilon}^k) = \mu \bm{\epsilon}$. To brevity, we denote $H^k=H(\bm{x}^k,\bm{\epsilon}^k)$, $H_{\bm{x}}^k=H_{\bm{x}}(\bm{x}^k,\bm{\epsilon}^k)$ and $H_{\bm{\epsilon}}^k=H_{\bm{\epsilon}}(\bm{x}^k,\bm{\epsilon}^k)$. Consequently, a natural idea is to use Anderson acceleration to accelerate the IRL1 algorithm. We propose the AAIRL1 algorithm. In each iteration, we apply Anderson acceleration to the mapping $H_{\bm{x}}$ to obtain a new iterate $\bm{x}^{k+1}$, and update $\bm{\epsilon}^{k+1}$ through the mapping $H_{\bm{\epsilon}}$. We present the framework of AAIRL1 algorithms in Algorithm \ref{algo1}, where we define $\bm{r}_1(\bm{x},\bm{\epsilon}) = H_{\bm{x}}(\bm{x},\bm{\epsilon})-\bm{x}$ and $\bm{r}_1^k =\bm{r}_1(\bm{x}^k,\bm{\epsilon}^k)$.

\begin{algorithm}[H]
    \renewcommand{\algorithmicrequire}{\textbf{Input:}}
    \renewcommand{\algorithmicensure}{\textbf{Initialize:}}
    \caption{The Proposed Anderson Accelerated IRL1 (AAIRL1)}\label{algo1}
    \begin{algorithmic}[1]
    \Require Given $\bm{x}^0 \in \mathbb{R}^n$ and $\bm{\epsilon}^0\in \mathbb{R}^n_{++}$. Pick $\mu \in (0,1)$ and $m \geq 1$.
    \Ensure Set $k = 0 $.
    \While{not converge}
    \State Set $m_k = \min(m,k)$
    \State Set $(H_{\bm{x}}^k,\bm{\epsilon}^{k+1}) \leftarrow H(\bm{x}^k,\bm{\epsilon}^k)$
    \State Set $R_1^k = [\bm{r}_1^k,...,\bm{r}_1^{k-m_k}]$, where $\bm{r}_1^{k} = H_{\bm{x}}^k - \bm{x}^k$ 
    \State Update $\bm{\alpha}^k \leftarrow \argmin_{\bm{\alpha}^T\bm{e}=1}\|R_1^k\bm{\alpha}\|_2$
    \State Update $\bm{x}^{k+1} = \sum^{m_k}_{i=0}\alpha^k_i H_{\bm{x}}^{k-m_k+i}$
    \State Set $k \gets k+1$
    \EndWhile
    \end{algorithmic}
\end{algorithm}

 Although Anderson acceleration does not employ derivatives in its iterations, establishing its convergence in the absence of continuous differentiability poses a significant challenge. The first convergence result was proposed in \cite[Theorem 2.3]{toth2015convergence}, which demonstrates that for a bounded $\sum^{m_k}_{i=0}|\alpha^k_i|$, given the continuous differentiability and contraction of $H$, Anderson acceleration can ensure the R-linear convergence rate of the fixed-point iteration when initiated near the fixed-point $\bm{x}^*$.
    
The theoretical analysis of \cite[Theorem 2.3]{toth2015convergence} requires the continuous differentiability of $H$. However, in the context of problem \eqref{equation1}, the mapping defined in (\ref{H}) struggles to meet this requirement. Notably, Wang et al. \cite{wang2021relating} suggested that IRL1 with $\ell_p$ norm is equivalent to gradient descent at the tail end of the algorithm in a smooth subspace. Inspired by this, we establish a local continuous differentiability of $H$.

\subsection{Local property of $H$}\label{subsec22}
Define $\bm{\bm{\theta}} = [\bm{x},\bm{\epsilon}]\in \mathbb{R}^n\times \mathbb{R}^n$. We begin by demonstrating that \(F(\bm{\bm{\theta}})\)  decreases monotonically over iterates \(\{\bm{\theta}^k\}\) that generated by $H$. Furthermore, we show $\{\bm{\theta}^k\}$ is bounded and establish the connection between the fixed-point of $H$ and the first-order stationary point of \eqref{equation1}.

\begin{Lemma} \label{Lemma1}
	Let Assumption \ref{Assumption1} and \ref{Assumption2} hold, then the following statements hold:
        \begin{itemize}\item [{\rm (i)}]
            The sequence $\{\bm{x}^k\}$ generated by $H$ is bounded, and ${F(\bm{\theta}^k)}$ is monotonically decreasing. Additionally, for all $k\in\mathbb{N}$, there exists a constant $C>0$ such that $\|\nabla Q_k(\bm{x}^{k+1})\|_{\infty}\leq C$.
            \item [{\rm (ii)}]
            Let $\bm{\theta}^* = [\bm{x}^*,0]$ be the cluster point of the sequence $\{\bm{\theta}^k\}$ that generated by $H$. $\bm{\theta}^* = [\bm{x}^*,0]$ is the fixed-point of $H$ and $\bm{x}^*$ satisfies the ﬁrst-order necessary optimality condition for \eqref{equation1}
            \end{itemize}
\end{Lemma}

\begin{proof}

(i)From Assumption \ref{Assumption1} and \ref{Assumption2}, we deduce that

\begin{equation*}
    \begin{aligned}
        f(\bm{x}^k) - f(\bm{x}^{k+1}) \geq \nabla f(\bm{x}^k)^T(\bm{x}^k-\bm{x}^{k-1})-\frac{L_f}{2}\|\bm{x}^k-\bm{x}^{k+1}\|^2,
\end{aligned}
\end{equation*}

\begin{equation*}
    \begin{aligned}
        Q_k(\bm{x}^k) - Q_k(\bm{x}^{k+1}) \geq \nabla f(\bm{x}^k)^T(\bm{x}^k-\bm{x}^{k-1})-\frac{M}{2}\|\bm{x}^k-\bm{x}^{k+1}\|^2.
\end{aligned}
\end{equation*}
Then it is straightforward to verify that

\begin{equation}
    \begin{aligned} \label{equation8}
        f(\bm{x}^k) - f(\bm{x}^{k+1})  \geq Q_k(\bm{x}^k) - Q_k(\bm{x}^{k+1})+\frac{M-L_f}{2}\|\bm{x}^k-\bm{x}^{k+1}\|^2.
\end{aligned}
\end{equation}
Furthermore, the concavity of \(\phi\) yields

\begin{equation}
    \begin{aligned} \label{equation9}
        \sum^n_{i=1}\phi(|x^{k+1}_i|+ \epsilon^k_i) \leq \sum^n_{i=1}\phi(|x^k_i|+\epsilon^k_i) + \sum^n_{i=1}\omega^k_i(|x^{k+1}_i| -|x^{k}_i|).
\end{aligned}
\end{equation}
Referring to equations (\ref{equation8}) and (\ref{equation9}):
\begin{equation}
    \begin{aligned} \label{equation10}
        f(\bm{x}^k,\bm{\epsilon}^k) - f(\bm{x}^{k+1},\bm{\epsilon}^{k+1})  \geq G_k(\bm{x}^k,\bm{\epsilon}^{k}) - G_k(\bm{x}^{k+1},\bm{\epsilon}^{k+1})+\frac{M-L_f}{2}\|\bm{x}^k-\bm{x}^{k+1}\|^2.
\end{aligned}
\end{equation}
From the ﬁrst-order necessary optimality condition (\ref{equation11}) of the subproblem (\ref{equation3}), we can derive

\begin{equation}
    \begin{aligned} \label{equation12}
        & G(\bm{x}^k;\bm{x}^k,\bm{\epsilon}^k) - G(\bm{x}^{k+1};\bm{x}^k,\bm{\epsilon}^k) \\
        =& Q_k(\bm{x}^k)-Q_k(\bm{x}^{k+1}) + \lambda \sum^n_{i=1}\omega^k_i(|x^{k}_i| -|x^{k+1}_i|) \\
        \geq & \nabla Q_k(\bm{x}^{k+1})(\bm{x}^k-\bm{x}^{k+1})+\frac{M}{2}\|\bm{x}^k-\bm{x}^{k+1}\|^2 + \lambda\sum^n_{i=1}\omega^k_i(|x^{k}_i| -|x^{k+1}_i|)  \\
        = & (\nabla Q_k(\bm{x}^{k+1}) + \lambda W^k \xi^k )^T(\bm{x}^k-\bm{x}^{k+1})+\frac{M}{2}\|\bm{x}^k-\bm{x}^{k+1}\|^2  \\
        = &  \frac{M}{2}\|\bm{x}^k-\bm{x}^{k+1}\|^2 .
    \end{aligned}
    \end{equation}
The first inequality follows from the convexity of \(Q_k(\cdot)\) and the absolute value inequality, the second inequality is a result of the strong convexity of \(Q_k(\cdot)\). By combining equations (\ref{equation10}) and (\ref{equation12}), we obtain
\begin{equation}
    \begin{aligned} \label{equation13}
        F(\bm{x}^k,\bm{\epsilon}^k) - F(\bm{x}^{k+1},\bm{\epsilon}^{k+1})  \geq (M-\frac{L_f}{2})\|\bm{x}^k-\bm{x}^{k+1}\|^2.
\end{aligned}
\end{equation}
Therefore, $\{F(\bm{x}^k,\bm{\epsilon}^k)\}$ is monotonically decreasing. Combining this with Assumption \ref{Assumption2} (i), we have $\forall k\in \mathbb{N}$, $-\infty < F(\bm{x}^k) \leq F(\bm{x}^k,\bm{\epsilon}^k) \leq F(\bm{x}^0,\bm{\epsilon}^0)$ and $\{\bm{x}^k\} \subset \mathcal{L}(F(\bm{x}^0,\bm{\epsilon}^0))$. By summing up the above inequality from $k=0$ to $k \rightarrow \infty$ we imply

\begin{equation}
    \begin{aligned} \label{equation14}
        \infty > F(\bm{x}^0,\bm{\epsilon}^0) - \displaystyle\lim_{k\rightarrow \infty}F(\bm{x}^{k+1},\bm{\epsilon}^{k+1})  \geq \sum^{\infty}_{i=0} (M-\frac{L_f}{2})\|\bm{x}^k-\bm{x}^{k+1}\|^2.
\end{aligned}
\end{equation}
Then, we have $\lim \limits_{k\rightarrow \infty}\|\bm{x}^{k+1}-\bm{x}^{k}\|=0$. The sequences $\{\bm{x}^k\}$ generated by $H$ is bounded. Furthermore, $\forall k\in\mathbb{N}, \exists C>0$ such that $\|\nabla Q_k(\bm{x}^{k+1})\|_{\infty} \leq C$.

(ii) We prove this by contradiction, define $\mathcal{A}^*=\{i:x^*_i=0\}, \mathcal{I}^*=\{i:x^*_i\ne0\}$ assume there exists $\hat{\bm{x}}$ such that $\hat{x}_i=0$ for all $i\in\mathcal{A}^*$ and $G(\bm{x}^*;\bm{x}^*;0)-G(\hat{\bm{x}} ;\bm{x}^*;0)>\varepsilon>0$. Suppose there exists a subsequence $\{\bm{x}^k\}_{S},S\in\mathbb{N}$ such that $x^k_i\xrightarrow{S}x^*_i$ and $\omega^k_i\xrightarrow{S}\omega^*_i$. Given that $\lim \limits_{k\rightarrow \infty}\|\bm{x}^{k+1}-\bm{x}^{k}\|=0$ and (\ref{equation12}), we can deduce that there exist $k_1>0$, such that for all $k>k_1$, it holds that 

\begin{equation}
    \begin{aligned} \label{A1}
G(\bm{x}^k;\bm{x}^k,\bm{\epsilon}^k) - G(\bm{x}^{k+1};\bm{x}^k,\bm{\epsilon}^k) \ge \varepsilon/4.
\end{aligned}
\end{equation}
There exist $k_2>0$, such that for all $k>k_2$ and $k\in S$, it holds that 
\begin{equation}
    \begin{aligned} \label{A2}
        (\nabla f(\bm{x}^*)-\nabla f(\bm{x}^k))\hat{\bm{x}}&>-\varepsilon/12,\\
        \sum_{i\in\mathcal{I}^*}\left[\omega(x^k_i,\epsilon^k_i)|x^k_i|-\omega(x^*_i,0)\right]|x^*_i|&>-\varepsilon/12,\\
        f(\bm{x}^k)-f(\bm{x}^*)&>-\varepsilon/12.\\
\end{aligned}
\end{equation}
Subsequently, we have

\begin{equation}
    \begin{aligned} \label{A3}
        & G(\bm{x}^*;\bm{x}^*,0) - G(\hat{\bm{x}};\bm{x}^k,\bm{\epsilon}^k) \\
        =& \left[Q_*(\bm{x}^*)+\lambda \sum_{i\in\mathcal{I}^*}\omega(x^*_i,0)|x^*_i|\right]\\
        &-\left[ \nabla f(\bm{x}^k)\hat{\bm{x}}+\frac{L}{2}\|\hat{\bm{x}}-\bm{x}^k\|^2 + \lambda \sum_{i\in\mathcal{I}^*}(\omega(x^*_i,0)-\omega(x^*_i,0)+\omega(x^k_i,\epsilon^k_i))(|\hat{x}_i|) \right] \\
        = & G(\bm{x}^*;\bm{x}^*,0)- G(\hat{\bm{x}};\bm{x}^*,0) \\
        &+ (\nabla f(\bm{x}^*)-\nabla f(\bm{x}^k))\hat{\bm{x}}+ \frac{L}{2}\|\hat{\bm{x}}-\bm{x}^*\|^2 + \lambda \sum_{i\in\mathcal{I}^*}(\omega(x^*_i,0)-\omega(x^k_i,\epsilon^k_i))(|\hat{x}_i|) \\
        \ge & \varepsilon-\frac{\varepsilon}{12}-\frac{\varepsilon}{12} = \frac{5\varepsilon}{6}.
\end{aligned}
\end{equation}
Besides, we can derive that

\begin{equation}
    \begin{aligned} \label{A4}
        & G(\bm{x}^k;\bm{x}^k,\bm{\epsilon}^k) - G(\bm{x}^*;\bm{x}^*,0) \\
        =& \left[Q_k(\bm{x}^k)+\lambda \sum_{i\in\mathcal{I}^*}\omega(x^k_i,\epsilon^k_i)|x^k_i| +\lambda \sum_{i\in\mathcal{A}^*}\omega(x^k_i,\epsilon^k_i)|x^k_i|\right]\\
        &-\left[ Q_*(\bm{x}^*) + \lambda \sum_{i\in\mathcal{I}^*}\omega(x^*_i,0)|x^*_i|\right] \\
        \ge & \left[Q_k(\bm{x}^k)+\lambda \sum_{i\in\mathcal{I}^*}\omega(x^k_i,\epsilon^k_i)|x^k_i|\right]\
        -\left[ Q_*(\bm{x}^*) + \lambda \sum_{i\in\mathcal{I}^*}\omega(x^*_i,0)|x^*_i|\right] \\
        \ge & -\frac{\varepsilon}{6}.
\end{aligned}
\end{equation}
Combining \eqref{A3} and \eqref{A4}, we can deduce that for all $k>\max(k_1,k_2), k\in S$, it holds that

\begin{equation}
\begin{aligned} \label{A6}
     G(\bm{x}^k;\bm{x}^k,\bm{\epsilon}^k) - G(\bm{x}^{k+1};\bm{x}^k,\bm{\epsilon}^k)
    &= G(\bm{x}^k;\bm{x}^k,\bm{\epsilon}^k) - G(\bm{x}^*;\bm{x}^*,0) + G(\bm{x}^*;\bm{x}^*,0) - G(\bm{x}^{k+1};\bm{x}^k,\bm{\epsilon}^k) \\
    &\ge \frac{5\varepsilon}{6} - \frac{\varepsilon}{6} = \frac{2\varepsilon}{3},
\end{aligned}
\end{equation}
contradicting \eqref{A1}. Therefore $\bm{\theta}^*$ is the fixed-point of $H$. According to a special case of Lemma 7 from \cite{wang2021nonconvex}, we can derive that $\bm{x}^*$ satisfies the ﬁrst-order necessary optimality condition of \eqref{equation1}. 
\end{proof}


Next, we proceed to analyze the local continuous differentiability of the mapping $H$. A notable property of the IRL1 is the locally stable sign property of the $\{\bm{x}^k\}$ for LPN approximation, which means that ${\rm sign}(\bm{x}^k)$ remain unchanged for sufficiently large $k$. At the tail of the iteration, IRL1 equates to solving a smooth problem in the reduced space $\mathbb{R}^{\mathcal{I}^*}$. We will utilize this characteristic to infer the local continuous differentiability of $H$. The non-degeneracy condition of $F$  is required(Define in Assumption \ref{Assumption1}). This is a common assumption in nonconvex problems. For LPN approximation, this condition naturally holds (due to $p(0^+)^{p-1} \rightarrow \infty$). For general sparse approximation, , the values of $\lambda$ and $p$ are often appropriately chosen to guarantee the sparsity. Consequently, the values of $\lambda$ or $\phi'(0^+)$ are large enough to satisfy the condition $\phi'(0^+) > C/\lambda$. Hence this assumption is easily satisfied. This assumption infers a lower bound property of $|x^*_i|$ away from 0 for all $i\in \mathcal{I}^*$. We have the following lemma.

    \begin{Lemma} \label{Lemma2} Let Assumption \ref{Assumption1} and \ref{Assumption2} hold. $H$ and fixed-point $\bm{x}^*$ hold the following statements.
        \begin{itemize}\item [{\rm (i)}]
            If there exists $(\overline{x}_i;\overline{\epsilon}_i)$ such that $\omega(\overline{x}_i,\overline{\epsilon}_i)> C/\lambda$, then for all $x_i\le\overline{x}_i$ and $\epsilon_i\le\overline{\epsilon}_i$, it holds that $[H_{\bm{x}}(\bm{x},\bm{\epsilon})]_i  = 0$.

            \item [{\rm (ii)}]
            There exists $\delta >0$ such that $\omega(\delta,0) = C/\lambda$. For all $i\in\mathcal{I}(\bm{x}^*)$, $|x^*_i|$ has the lower bound, $|x^*_i|\geq \delta.$
            \end{itemize}
    \end{Lemma}

\begin{proof}
    (i) We prove this by contradiction. Assume that we have $\omega(\overline{x}_i,\overline{\epsilon}_i) > C/\lambda$. If $[H_{\bm{x}}(\overline{x},\overline{\epsilon})]_i  \neq 0$, it would contradict the ﬁrst-order necessary optimality condition (\ref{equation11}). Therefore, there must have $[H(\overline{x},\overline{\epsilon})]_i = 0$. Furthermore, for $x_i\le\overline{x}_i$ and $\epsilon_i\le\overline{\epsilon}_i$, we have

\begin{equation}
	\begin{aligned} \label{equation16}
		\omega(x_i,\epsilon_i) 
        = \phi'(|x_i|+\epsilon_i)
        \geq \phi'(|\overline{x}_i|+\overline{\epsilon}_i)
        = \omega(\overline{x}_i,\overline{\epsilon}_i) 
        > \frac{C}{\lambda}.
\end{aligned}  
\end{equation}
Consequently, we have $ [H(\bm{x},\bm{\epsilon})]_i = 0$.

(ii) For $i\in\mathcal{I}^*$, it holds that $[H_{\bm{x}}(\bm{x}^*,0)]_i = x^*_i\ne 0$. Hence, we can deduce that $x^*_i\geq \delta$.

\end{proof} 

Based on the boundness of $|x^*_i|$ for all $i\in \mathcal{I}^*$, we establish the locally continuous differentiability of $H$:

\begin{Lemma}\label{Lemma3} 
    Let Assumption \ref{Assumption1} and \ref{Assumption2} hold. There exists a $ \rho < \min (\delta,\mathop{{\rm min}}_{i\in\mathcal{I}^*} x^*_i-\delta)$. For all $\bm{\theta}\in\mathcal{B}(\rho):=\{\bm{\theta}\in \mathbb{R}^{2n} | \|\bm{\theta}-\bm{\theta}^*\|\leq \rho\}$, $ {\rm sign}(H_{\bm{x}}(\bm{\theta})) = {\rm sign}(\bm{x}) $ and $H(\bm{\theta})$ is continuously differentiable.
\end{Lemma}

\begin{proof} The continuity of $H$ infers that there exists a $\rho$ such that $\|H(\bm{\theta})-H(\bm{\theta}^*)\| = \|H(\bm{\theta})-\bm{\theta}^*\| < \delta$ for all $\bm{\theta} \in \mathcal{B}(\rho)$.

    For $i\in\mathcal{A}^*$, given that $ \rho < \min (\delta,\mathop{{\rm min}}_{i\in\mathcal{I}^*} x^*_i-\delta)$, it is established that $\omega(\bm{x},\bm{\epsilon})> C/\lambda$ and $[H_{\bm{x}}(\bm{\theta})]_i=0$. Consequently, we can deduce that ${\rm sign}([H_{\bm{x}}(\bm{\theta})]_i)= {\rm sign}(x^*_i)$.  
    
    For $i\in\mathcal{I}^*$, we prove by contradiction, assume that ${\rm sign}([H_{\bm{x}}(\bm{\theta})]_i)\neq {\rm sign}(x^*_i)$. Therefore, we can deduce that
\begin{equation}
	\begin{aligned} \label{equation20}
		|[H_{\bm{x}}(\bm{\theta})]_i-x^*_i| =  \sqrt{([H_{\bm{x}}(\bm{\theta})]_i)^2+(x^*_i)^2-2[H_{\bm{x}}(\bm{\theta})]_ix^*_i}>\sqrt{(x^*_i)^2}\geq \sqrt{(\delta)^2} = \delta.
\end{aligned}
\end{equation}
This is in contradiction with $\|H(\bm{\theta})-\bm{\theta}^*\| < \delta$. As a result, we can infer that ${\rm sign}([H_{\bm{x}}(\bm{\theta})]_i) = {\rm sign}(x^*_i)$ for all $i\in\mathcal{I}^*$ and $ {\rm sign}(H_{\bm{x}}(\bm{\theta})) = {\rm sign}(\bm{x}) $.

Given that for all $\bm{\theta} \in \mathcal{B}(\rho)$,  $H(\bm{\theta})$ is not continuous differentiable only when ${\rm sign}(x_i)\neq {\rm sign}(x^*_i)$, $ i\in \mathcal{I}^*$, we can conclude that $H(\bm{\theta})$ is continuously differentiable in $\mathcal{B}(\rho)$. 

\end{proof}

\subsection{Local convergence guarantees}

Subsequently, we analyze the local convergence of the algorithm. We make the following assumption.
\begin{Assumption} \label{Assumption3} There exists  $0< \hat{\rho} \le \rho$ such that for all $\bm{\theta} \in \mathcal{B}(\hat{\rho})$, the following statement hold: 
	\begin{itemize} \item [{\rm (i)}] There exists $\kappa  > 0$ such that $\nabla ^2 F([\bm{x}_{\mathcal{I}^*};\bm{\epsilon}_{\mathcal{I}^*}]) \succeq \kappa I$.
    \item [{\rm (ii)}]
    For all $i \in \mathcal{I}^*$, $\omega(\theta_i)$ is Lipschitz continuous with constant $L_w>0$.
	\item [{\rm (iii)}]
    For all $\bm{\theta} $ and $ \bm{\theta}^k \in \mathcal{B}(\hat{\rho})$, there exists $C >0$ and such that $\|\nabla Q_k(x)\|_\infty \leq C$.
    \item [{\rm (iv)}]
    For all $k \in \mathbb{N}$, there exists an upper bound $ M_{\bm{\alpha}}$ such that $ \sum^{m_k}_{i=0}|\alpha^k_i| \leq M_{\bm{\alpha}}$.
    \end{itemize}
\end{Assumption}

Recent studies have demonstrated the capability of IRL1 algorithm to avoid active strict saddle points and furthermore converge to a local minimum. Therefore, we make Assumption \ref{Assumption3} (i). It is crucial to ensure that the mapping $H$ is a contraction. Subsequently, Assumption \ref{Assumption3} (ii) and (iii) are straightforward in the text of problem \eqref{equation1}. (ii) is naturally satisfied for $x_i\neq 0$. (iii) assumes a upper bound of $|\nabla \omega(\theta_i)|$ and $\|\nabla Q_k(\bm{x})\|_\infty$. It is a similar assumption in Assumption \ref{Assumption2} and which is trivially holds in the local of $\bm{x}^*$. Additionally, Assumption \ref{Assumption3} (iv) is commonly found in the existing literature of Anderson acceleration. It ensures the boundness of $ \sum^{m_k}_{i=0}|\alpha^k_i|$, In experiments, we have not observed a case that the coefficients become large. However, the relevant proof of boundness has not yet been proposed. There are several practical method can be used to enforce it \cite{scieur2016regularized,toth2015convergence}.

In the following, we demonstrate that $H(\bm{\theta})$ is a contraction when $\bm{\theta}$ is close to $\bm{\theta}^*$. Denote $\nabla H$ as the Jacobian matrix of $H$.

\begin{Lemma}\label{Lemma4} Let Assumption \ref{Assumption1}, \ref{Assumption2}, \ref{Assumption3} hold and $L> \max (\kappa, L_{\omega})$. Then, $H(\bm{\theta})$ is a contraction mapping for all $\bm{\theta}\in \mathcal{B}(\hat{\rho})$, i.e. For all $\bm{\theta}_1,\bm{\theta}_2 \in\mathcal{B}(\hat{\rho})$, there exists $\gamma\in(0,1)$ such that

    \begin{equation}\label{equation22}
        \|H(\bm{\theta}_1) - H(\bm{\theta}_2)\| \leq \gamma\|\bm{\theta}_1-\bm{\theta}_2\|.
    \end{equation}
\end{Lemma}

\begin{proof} Given that $H_{\bm{\epsilon}}(\bm{\theta}) = \mu \epsilon $, it is easy to deduce that $\|\nabla H_{\bm{\epsilon}}(\bm{\theta})\| \le \mu < 1 $. 

    In the Lemma \ref{Lemma3}, we show that $ {\rm sign}(H_{\bm{x}}(\bm{\theta})) = {\rm sign}(\bm{x}) $ and $H(\bm{\theta})$ is continuously differentiable for all $\bm{\theta}\in \mathcal{B}(\hat{\rho})$. Therefore, it infers that for $i \in \mathcal{A}^*$, we have $[H_{\bm{x}}(\bm{x},\bm{\epsilon})]_i=0$. Then it holds that
    $$\|\nabla H_{\bm{x}}([\bm{x}_{\mathcal{A}^*};\bm{\epsilon}_{\mathcal{A}^*}])\| = 0 < 1.$$
For $i\in\mathcal{I}^*$, we can derive that 
$$\nabla H_{\bm{x}}([\bm{x}_{\mathcal{I}^*};\bm{\epsilon}_{\mathcal{I}^*}]) = 
\left[\begin{array}{c}
      I - \frac{1}{L}\nabla^2F([\bm{x}_{\mathcal{I}^*};\bm{\epsilon}_{\mathcal{I}^*}]) \\
      -\frac{1}{L}\nabla^2 \Phi([\bm{x}_{\mathcal{I}^*};\bm{\epsilon}_{\mathcal{I}^*}]) 
    \end{array}\right]. $$
 Drawing upon Assumption \ref{Assumption3} (i) and (ii), we have 

 $$\|\nabla ^2 F([\bm{x}_{\mathcal{I}^*};\bm{\epsilon}_{\mathcal{I}^*}])  \| \ge \kappa,   \:\: \|\nabla^2 \Phi([\bm{x}_{\mathcal{I}^*};\bm{\epsilon}_{\mathcal{I}^*}]  \| \le L_{\omega}.$$
With combination of $L> \max (\kappa, L_{\omega})$, we can deduces that

$$\|\nabla H_{\bm{x}}([\bm{x}_{\mathcal{I}^*};\bm{\epsilon}_{\mathcal{I}^*}])\| \le \max(1-\kappa/L,L_w/L)<1.$$ 
Therefore, $H(\bm{\theta})$ is a contraction mapping for all $\bm{\theta}\in$  $\mathcal{B}(\hat{\rho})$ and (\ref{equation22}) holds.

\end{proof}

In the following, we present a guarantee of local R-linear convergence for the AAIRL1 algorithm.

\begin{Theorem} \label{Theorem1} Let Assumption \ref{Assumption1}, \ref{Assumption2}, \ref{Assumption3} hold and $L> \max (\kappa, L_{\omega})$. Let $\bm{\theta}^*=H(\bm{\theta}^*)$ be a fixed-point of $H$, then if  $\bm{\theta}^0$ is sufficiently close to $\bm{\theta}^*$, the iterates $\{\bm{\theta}^k\}$ generated by AAIRL1 algorithm converge to $\bm{\theta}^*$ R-linearly with $\hat{\gamma}\in(\gamma,1)$.
    \begin{equation}
        \begin{aligned} \label{equation17}
            \|H^k-\bm{\theta}^k\| \leq \hat{\gamma}^k\|H^0-\bm{\theta}^0\|
    \end{aligned}
    \end{equation}
and
    \begin{equation}
        \begin{aligned} \label{equation18}
            \|\bm{\theta}^k-\bm{\theta}^*\| \leq \frac{1+\gamma}{1-\gamma} \hat{\gamma}^k\|\bm{\theta}^0-\bm{\theta}^*\|.
    \end{aligned}
    \end{equation}

\end{Theorem}

\begin{proof}


Lemma \ref{Lemma3} implies that $H(\bm{x},\bm{\epsilon})$ is continuously differentiable for $\bm{\theta}\in\mathcal{B}(\rho)$. Then obviously $\bm{r}$ is also continuously differentiable. Denote $L_{\bm{r}}$ is the Lipschitz constant of $\nabla \bm{r} $ in $\mathcal{B}(\rho)$. There we need a special case of the result in \cite{kelly1995iterative} (Lemma 4.3.1). Denote $\bm{e} = \bm{\theta} - \bm{\theta}^*$. When $\rho \leq \hat{\rho}$ sufficiently small, for all $\bm{\theta} \in\mathcal{B}(\rho)$, we deduce

\begin{equation}\label{-equation5}
    \|\bm{r}(\bm{\theta}) - \nabla \bm{r}(\bm{\theta}^*)\bm{e}\| \leq \frac{L_{\bm{r}}}{2}\|\bm{e}\|^2
\end{equation}
and 

\begin{equation}\label{-equation6}
    \|\bm{e}\|(1-\gamma) \leq \|\bm{r}(\bm{\theta})\| \leq \|\bm{e}\|(1+\gamma).
\end{equation}
First we reduce $\rho$ then $\rho<2(1-\gamma)/L_{\bm{r}}$ and 

\begin{equation}\label{-equation9}
    \frac{\left(1 + \left(\frac{ M_{\bm{\alpha}}L_{\bm{r}}\rho}{2(1-\gamma)}\right)\gamma^{-m-1}\right)}{\left(1-\frac{L_{\bm{r}}\rho}{2(1-\gamma)}\right)}\leq 1.
\end{equation}
Reduce $\|\bm{e}^0\|$ make 
\begin{equation}\label{-equation10}
    \left(\frac{ M_{\bm{\alpha}}(\gamma+L_{\bm{r}}\rho/2)}{2(1-\gamma)}\right)\gamma^{-m}\|\bm{r}(\bm{x}^0)\| \leq \left(\frac{ M_{\bm{\alpha}}(1+\gamma)(\gamma+L_{\bm{r}}\rho/2)}{2(1-\gamma)}\right)\gamma^{-m}\|\bm{e}^0\| \leq \rho .
\end{equation}
Now we prove by induction, assume for all $0\leq k \leq K$
\begin{equation}\label{-equation11}
    \|\bm{r}(\bm{\theta}^K)\|\leq \gamma^k \|\bm{r}(\bm{\theta}^0)\|
\end{equation}
and

\begin{equation}\label{-equation12}
    \|\bm{e}^k\|\leq \rho.
\end{equation}
It is obviously when $K=0$. From equation (\ref{-equation5}), denote

\begin{equation}\label{-equation13}
    \bm{r}(\bm{\theta}^k) = \nabla \bm{r}(\bm{\theta}^0)\bm{e}^k + \Delta^k,
\end{equation}
where 

\begin{equation}\label{-equation14}
    \| \Delta^k \| \leq \frac{L_{\bm{r}}}{2}\|\bm{e}^{k}\|^2.
\end{equation}
Therefore, it holds that

\begin{equation}\label{-equation15}
    H(\bm{x}^k,\bm{\epsilon}^k) = \bm{\theta}^* + H'(\bm{\theta}^*)\bm{e}^k + \Delta^k.
\end{equation}
Weighted summing (\ref{-equation15}) with $\bm{\alpha}^k$ ($\sum \alpha^K_j = 1$), we have

\begin{equation}\label{-equation16}
    \begin{aligned}
        \left[\begin{array}{c}
            \bm{x}^{k+1}  \\
            \sum^{m_K}_{j=0}\alpha^K_j \bm{\epsilon}^{K-m_K+j}
            \end{array}\right]
=  \bm{\theta}^* + \sum^{m_K}_{j=0}\alpha^K_j(H'(\bm{x}^*,0)\bm{e}^{K-m_K+j}) + \sum^{m_K}_{j=0}\alpha^K_j\Delta^{K-m_K+j}.
    \end{aligned}
\end{equation}
Given that $\|\bm{\epsilon}^{K+1}\|\le \|\sum^{m_K}_{j=0}\alpha^K_j \bm{\epsilon}^{K-m_K+j}\|$, it implies that

\begin{equation}\label{-equation17}
        \|\bm{e}^{K+1}\| \le  \|\sum^{m_K}_{j=0}\alpha^K_j(H'(\bm{x}^*,0)\bm{e}^{K-m_K+j}) + \sum^{m_K}_{j=0}\alpha^K_j\Delta^{K-m_K+j}\|.
\end{equation}
Next, we analysis the boundness of $\sum^{m_K}_{j=0}\alpha^K_j\Delta^{K-m_K+j}$:

\begin{equation}\label{-equation18}
    \begin{aligned}
    \|\sum^{m_K}_{j=0}\alpha^K_j\Delta^{K-m_K+j} \| 
    &\leq  \sum^{m_K}_{j=0}|\alpha^k_j|\frac{L_{\bm{r}}}{2}\|\bm{e}^{K-m_K+j}\|^2  \\
    & \leq \sum^{m_K}_{j=0} |\alpha^K_j|\frac{L_{\bm{r}}\rho}{2(1-\gamma)}\|\bm{r}^{K-m_K+j}\| \\
    & \leq \sum^{m_K}_{j=0} |\alpha^K_j|\frac{L_{\bm{r}}\rho}{2(1-\gamma)} \gamma^{K-m_K+j} \|\bm{r}^{0}\| \\
    & \leq\frac{ M_{\bm{\alpha}}L_{\bm{r}}\rho}{2(1-\gamma)}\gamma^{-m} \|\bm{r}^{0}\| .
\end{aligned}
\end{equation}
The first inequality is by (\ref{-equation14}), the second inequality is by (\ref{-equation12}) and (\ref{-equation6}), the third inequality is by (\ref{-equation11}), the last inequality is by $\sum |\bm{\alpha}^{K}| \leq M_{\bm{\alpha}}$ and $K-m^{K}+j\geq -m$. Similarly, we can get 

\begin{equation}\label{-equation19}
\| \sum^{m_K}_{j=0}\alpha^K_j(H'(\bm{x}^*,0)\bm{e}^{K-m_K+j}) \| \leq \frac{M_{\bm{\alpha}}\gamma}{1-\gamma}\gamma^{-m}\|\bm{r}^{0}\|.
\end{equation}

(\ref{-equation17}),(\ref{-equation18}) and (\ref{-equation19}) imply that

\begin{equation}\label{-equation20}
\| \bm{e}^{K+1}\| \leq \frac{M_{\bm{\alpha}}(\gamma+L_{\bm{r}}\rho/2)}{1-\gamma}\gamma^{-m} \|\bm{r}^0\| \leq \rho.
\end{equation}
Following this, we can apply (\ref{-equation13}) with $k = K+1$

\begin{equation}\label{-equation21}
    \begin{aligned}
        \bm{r}(\bm{\theta}^{K+1}) &=\nabla \bm{r}(\bm{\theta}^*)\bm{e}^{K+1} + \Delta^{K+1} \\
        & = (H'(\bm{\theta}^*)-I)\bm{e}^{K+1} + \Delta^{K+1},
\end{aligned}
\end{equation}
and 

\begin{equation}\label{-equation22}
\| \Delta^{K+1} \| \leq \frac{L_{\bm{r}}}{2}\|\bm{e}^{K+1} \|^2.
\end{equation}
Subsequently, by (\ref{-equation6}), (\ref{-equation20}) and (\ref{-equation22}), we can obtain that

\begin{equation}\label{-equation23}
\| \Delta^{K+1} \| \leq \frac{L_{\bm{r}}\rho}{2(1-\gamma)}\|\bm{r}^{K+1} \|.
\end{equation}
Since $H'(\bm{x}^*,0)$ and $H'(\bm{x}^*,0)-I$ commute, then 

\begin{equation}\label{-equation24}
\begin{aligned}
    \|\bm{r}(\bm{\theta}^{K+1})\|= &\|(H_{\bm{x}}'(\bm{x}^*,0)-I)\bm{e}^{K+1} + \Delta^{K+1}\| \\
        \le & \|H_{\bm{x}}'(\bm{x}^*,0)\left(\sum^{m_K}_{j=0}\alpha^K_j((H_{\bm{x}}'(\bm{x}^*,0)-I)\bm{e}^{K-m_K+j})\right) \\
    &+ (H_{\bm{x}}'(\bm{x}^*,0)-I)\sum^{m_K}_{j=0}\alpha^K_j\Delta^{K-m_K+j} + \Delta^{K+1}\| \\
    = & \|H_{\bm{x}}'(\bm{x}^*,0)\left(\sum^{m_K}_{j=0}\alpha^K_j((H_{\bm{x}}'(\bm{x}^*,0)-I)\bm{e}^{K-m_K+j}+\sum^{m_K}_{j=0}\alpha^K_j\Delta^{K-m_K+j})\right) \\
    &-\sum^{m_K}_{j=0}\alpha^K_j\Delta^{K-m_K+j} + \Delta^{K+1} \|\\
        = &\| H_{\bm{x}}'(\bm{x}^*,0)\sum^{m_K}_{j=0}\alpha^K_j \bm{r}^{K-m_K+j}-\sum^{m_K}_{j=0}\alpha^K_j\Delta^{K-m_K+j} + \Delta^{K+1}\|, \\
\end{aligned}
\end{equation}
where first and last equality is by (\ref{-equation21}), second equality is by (\ref{-equation17}). Hence

\begin{equation}\label{-equation25}
\begin{aligned}
    \left(1-\frac{L_{\bm{r}}\rho}{2(1-\gamma)}\right)\|\bm{r}^{K+1}\|\leq  & \|\bm{r}^{K+1}\| - \|\Delta^{K+1}\| \\
        \leq & \gamma\|\sum^{m_K}_{j=0}\alpha^K_j \bm{r}^{K-m_K+j} \| + \|\sum^{m_K}_{j=0}\alpha^K_j\Delta^{K-m_K+j}\|  \\
        \leq & \gamma\|\bm{r}^{K}\| + \|\sum^{m_K}_{j=0}\alpha^K_j\Delta^{K-m_K+j}\| \\
        =& \left(1 + \left(\frac{ M_{\bm{\alpha}}L_{\bm{r}}\rho}{2(1-\gamma)}\right)\gamma^{-m-1}\right)\gamma^{K+1} \|\bm{r}^0\|,
\end{aligned}
\end{equation}
where the first inequality is by (\ref{-equation23}), second inequality is by (\ref{-equation24}), third inequality is because $\bm{\alpha}^k \leftarrow \mathop{argmin}_{\bm{\alpha}^T\textbf{1}=1}\|R^k\bm{\alpha}\|$ ($R^k$ is define in Algorithm 1) and the first equality is by (\ref{-equation18}). Therefore, from (\ref{-equation9}) we can deduce that

\begin{equation}\label{-equation26}
    \begin{aligned}
        \|\bm{r}(\bm{x}^{k+1})\| \leq   \frac{\left(1 + \left(\frac{ M_{\bm{\alpha}}L_{\bm{r}}\rho}{2(1-\gamma)}\right)\gamma^{-m-1}\right)}{\left(1-\frac{L_{\bm{r}}\rho}{2(1-\gamma)}\right)} \gamma^{K+1}\|\bm{r}(\bm{x}^0)\| \leq \gamma^{K+1}\|\bm{r}(\bm{x}^0)\|.
\end{aligned}
\end{equation}
Then (\ref{equation18}) hold.

\end{proof}

\section{ Globally convergent Anderson accelerated IRL1}

In the previous section, we established that under mild conditions, if the initial point is close to the fixed-point, the AAIRL1 algorithm has a remarkable convergence rate. However, its global convergence still remains largely uncertain. Even when dealing with strongly convex objective functions, Anderson accelerated gradient descent algorithm struggles to guarantee global convergence. Numerous studies have indicated that AA methods do not converge globally under more relaxed conditions \cite{scieur2016regularized,walker2011anderson}. Many existing works employ safe guarding strategies \cite{zhang2020globally,fu2020anderson,mai2020anderson} for AA methods to ensure global convergence.

Motivated by these studies, “a logical approach is to incorporate a safety condition to the AA step in each iteration. If the safety condition is satisfied, the AA step will be accepted to update the new iterate; otherwise, the unaccelerated step is accepted. By alternating between these two types of steps, global convergence is ensured. To increase the performance of our algorithm, we incorporated the idea of nonmonotone line search \cite{grippo1986nonmonotone} into our algorithm. Nonmonotone line search has been successfully applied in various optimization problems \cite{liuzzi2020algorithmic,mita2019nonmonotone,ferreira2023subgradient}. Consequently, we introduce a relaxed nonmonotone line search condition \cite{zhang2004nonmonotone}, it relaxes the requirement of a strictly decreasing function value, allowing for a higher acceptance rate of AA steps while ensuring algorithm convergence. The resulting globally convergent AAIRL1 algorithm is as follows:
 \begin{algorithm}[H]
    \renewcommand{\algorithmicrequire}{\textbf{Input:}}
    \renewcommand{\algorithmicensure}{\textbf{Initialize:}}
    \caption{Guard AAIRL1}\label{algo2}
        \begin{algorithmic}[1]
        \Require Given $ \bm{x}^0\in\mathbb{R}^n$ and $\bm{\epsilon}^0\in \mathbb{R}^n_{++}$. Pick $\beta,\mu \in (0,1)$, $\eta \in [0,1]$,   $m\ge 1$.
        \Ensure Set $k = 0 $, $E^0 =F(\bm{x}^0,\bm{\epsilon}^0) $ and $J^0 = 1$. 
        \While{not convergence}
        \State Set  $m_k = \min (m,k)$
        
        \State Set	$(H_{\bm{x}}^k,\bm{\epsilon}^{k+1}) \leftarrow H(\bm{x}^k,\bm{\epsilon}^k)$
        
        \State Set $R^k_1 = [\bm{r}^k_1,...,\bm{r}^{k-m_k}_1]$, where $\bm{r}^k_1 = H_{\bm{x}}^k - \bm{x}^k$

        \State Update $\bm{\alpha}^k \leftarrow \mathop{argmin}_{\bm{\alpha}^T\textbf{e}=1}\|R^k_1\bm{\alpha}\|$
        
        \State Set $\bm{x}^{k+1}_{AA} = \sum^{m_k}_{i=0}\alpha^k_i H_{\bm{x}}^{k-m_k+i}$

		\State Set $\chi(\bm{x}^k,\bm{\epsilon}^k) =  \mathop{{\rm max}}\limits_{i = 1,...,n}{\rm dist}(-\nabla_i f(\bm{x}^k), \omega(x_i^k,\epsilon_i^k)\partial|x_i^k| )$

        \If {$F(\bm{x}^{k+1}_{AA},\bm{\epsilon}^{k+1})\leq E^k - \beta\chi(\bm{x}^k,\bm{\epsilon}^k)$}

        \State Update $\bm{x}^{k+1} = \bm{x}^{k+1}_{AA}$
        \Else
        \State Update $\bm{x}^{k+1} = H_{\bm{x}}^k$
        \EndIf
        \State Update $J^{k+1} = \eta J^{k}+1$
        \State Update $E^{k+1} = (\eta J^{k}E^{k}+F(\bm{x}^{k+1},\bm{\epsilon}^{k+1}))/J^{k+1}$
        \State Set $k \leftarrow k+1$
        \EndWhile
        \end{algorithmic}
\end{algorithm}
In the algorithm \ref{algo2}, we define $E^{k}$ as a convex combination of $F(\bm{x}^{0},\bm{\epsilon}^{0}),F(\bm{x}^{1 },\bm{\epsilon}^{1}), ..., F(\bm{x}^{k},\bm{\epsilon}^{k})$, with more weights biased towards the newer $F$ and $E^{0}=F(\bm{x}^{0},\bm{\epsilon}^{0})$. Straightforwardly, $E^{k+1}$ is a convex combination of 
$E^{k}$ and $F(\bm{x}^{k+1},\bm{\epsilon}^{k+1})$  $\chi(\bm{x},\bm{\epsilon})$ serves as a measure for the ﬁrst-order necessary optimality of \eqref{equation1} at $(\bm{x},\bm{\epsilon})$. 

\begin{equation}
	\begin{aligned} \label{measure optimal}
		\chi(\bm{x},\bm{\epsilon}) :=  \mathop{{\rm max}}\limits_{i = 1,...,n}{\rm dist}(-\nabla_i f(x), \omega(x_i,\epsilon_i)\partial|x_i| ).
\end{aligned}
\end{equation}

\begin{equation}
	\begin{aligned} \label{nonmonotone descent condition}
		F(\bm{x}^{k+1}_{AA},\bm{\epsilon}^{k+1})\leq E^k - \beta\chi(\bm{x}^k,\bm{\epsilon}^k).
\end{aligned}
\end{equation}

In each iteration of Algorithm \ref{algo2}, we use the nonmonotone descent condition (\ref{nonmonotone descent condition}) as a criterion for accepting the AA step. If the condition (\ref{nonmonotone descent condition}) is satisfied, it implies sufficient descent has been achieved, and the AA step will be accepted. Conversely, when the condition (\ref{nonmonotone descent condition}) is not fulfilled, the unaccelerated step will be chosen.

 \subsection{Global convergence of Guard AAIRL1}

 This section provides a proof of the global convergence of the Algorithm \ref{algo2}. Initially, we can notice that \(\bm{x}^{k+1}\) is obtained through either the AA step or the unaccelerated step. Therefore, we can divide the iteration counts into two sets: \(K_{AA} := \{k_0, k_1, \ldots\}\), which denote the iterations where the AA step is accepted, and \(K_{UA} := \{l_0, l_1, \ldots\}\), which denote the unaccelerated steps. Obviously, $K_{AA}\bigcup K_{UA} = \mathbb{N} $

 \begin{Lemma}\label{Lemma5} Let Assumption \ref{Assumption1}, \ref{Assumption2} hold, it follows that the sequence $\{\bm{x}^k\}$ generated by Algorithm \ref{algo2} is bounded.
\end{Lemma}

\begin{proof}

    Initially, we prove that $F(\bm{x}^{k},\bm{\epsilon}^{k})\leq E^k$ and $E^{k+1}\leq E^{k}$ for all $k\in \mathbb{N}$, we prove by induction. Assume that $F(\bm{x}^{k},\bm{\epsilon}^{k})\leq E^k$, it is obviously holds when $k=0$. Let's assume the inequality holds for some $k$, i.e., $F(\bm{x}^{k},\bm{\epsilon}^{k})\leq E^k$. Consequently we  proceed to discuss different condition. 

    For $k_i\in K_{AA}$, we have $F(\bm{x}^{k_i+1},\bm{\epsilon}^{k_i+1})\leq E^{k_i}$, as indicated by (\ref{nonmonotone descent condition}). Moreover, since $E^{{k_i}+1}$ is a convex combination of $E^{{k_i}}$ and $F(\bm{x}^{{k_i}+1},\bm{\epsilon}^{{k_i}+1})$, we can infer that $F(\bm{x}^{{k_i}+1},\bm{\epsilon}^{{k_i}+1})\leq E^{{k_i}+1} \leq E^{{k_i}}$.

    For $l_i\in K_{UA}$, similar to (\ref{equation13}), it follows that
    \begin{equation}
        \begin{aligned} 
             F(\bm{x}^{l_i+1},\bm{\epsilon}^{l_i+1})  
             &\leq F(\bm{x}^{l_i},\bm{\epsilon}^{l_i}) - \frac{M-L_f}{2}\|\bm{x}^{l_i}-\bm{x}^{l_i+1}\|^2 \\
             &\leq E^{l_i} - \frac{M-L_f}{2}\|\bm{x}^{l_i}-\bm{x}^{l_i+1}\|.
    \end{aligned}
    \end{equation}
    Similarly, we can deduce that  $F(\bm{x}^{l_i+1},\bm{\epsilon}^{l_i})\leq E^{l_i+1} \leq E^{l_i}$. Therefore, it holds that $F(\bm{x}^{k+1},\bm{\epsilon}^{k+1})\leq E^{k+1}$. $E^{k+1}\leq E^{k}$ and $E^k$ is monotonically non-increasing. Hence we have $-\infty \leq F(\bm{x}^{k},\bm{\epsilon}^{k}) \leq E^k\leq E^0 = F(\bm{x}^0,\bm{\epsilon}^0)$ for all $k\in\mathbb{N} $ and  $\{\bm{x}^k\}\subset \mathcal{L}(F(\bm{x}^0,\bm{\epsilon}^0))$. The sequence $\{\bm{x}^k\}$ generated by Algorithm \ref{algo2} is bounded.
\end{proof}

 \begin{Theorem} \label{Theorem2} Let Assumption \ref{Assumption1},\ref{Assumption2} hold and $\bm{x}^*$ is an accumulation point of $\{\bm{x}^{k}\}$ that generated by Algorithm \ref{algo2}, it follows that  $\bm{x}^*$ is first-order necessary optimal for \eqref{equation1}

\begin{proof}
For $k_i\in K_{AA}$, it follows that 
\begin{equation}
    \begin{aligned}\label{AA descent}
            E^{k_i+1}  
            &= \frac{\eta J^{k_i}E^{k_i}+F(\bm{x}^{k_i+1},\bm{\epsilon}^{k_i+1})}{J^{k_i+1}} \\
            &\leq \frac{\eta J^{k_i}E^{k_i}+E^k_i - \beta\chi(\bm{x}^{k_i},\bm{\epsilon}^{k_i})}{J^{k_i+1}} = E^{k_i} - \frac{\beta\chi(\bm{x}^{k_i},\bm{\epsilon}^{k_i})}{J^{k_i+1}} .
\end{aligned}
\end{equation}
For $l_i\in K_{UA}$, it follows that 

\begin{equation}
    \begin{aligned}\label{UA descent}
            E^{l_i+1}  
            &= \frac{\eta J^{l_i}E^{l_i}+F(\bm{x}^{l_i+1},\bm{\epsilon}^{l_i+1})}{J^{l_i+1}} \\
            &\leq \frac{\eta J^{l_i}E^{l_i}+E^{l_i} - (M-L_f)\|\bm{x}^{l_i}-\bm{x}^{l_i+1}\|^2/2} {J^{l_i+1}} \\
            &= E^{l_i} - \frac{(M-L_f)\|\bm{x}^{l_i}-\bm{x}^{l_i+1}\|^2} {2J^{l_i+1}}.
\end{aligned}
\end{equation}
Summing up (\ref{AA descent}) and (\ref{UA descent}) for all $k_i\in K_{AA}$ and $l_i\in K_{UA}$, we can deduce that

\begin{equation}
    \begin{aligned}\label{bounded}
        \infty > E^0 - \displaystyle\lim_{k\rightarrow \infty}E^{k+1}  \geq \sum_{k_i \in K_{AA}}\frac{\beta\chi(\bm{x}^{k_i},\bm{\epsilon}^{k_i})}{J^{k_i+1}}  +  \sum_{l_i \in K_{UA}}\frac{(M-L_f)\|\bm{x}^{l_i}-\bm{x}^{l_i+1}\|^2} {2J^{l_i+1}}.
\end{aligned}
\end{equation}


Due to the boundness of $\mathcal{L}(F(\bm{x}^0))$, we have the sequence $\{\bm{x}^k\}$ is bounded. Let $\bm{x}^*$ be a cluster point of $\{\bm{x}^k\}$ with subsequence $\{\bm{x}^k\}_S$. Subsequently, when $k\rightarrow \infty$, one of the following statements holds:

\begin{equation}
    \begin{aligned}\label{Accelerated infty}
        \lim_{\substack{k\rightarrow \infty \\ k\in S}    } \chi(\bm{x}^{k},\bm{\epsilon}^{k})= \lim_{\substack{k\rightarrow \infty \\ k\in S}} {\rm dist}(-\nabla f(\bm{x}^{k}), W^{k}\partial|\bm{x}^{k}| ) = 0.
\end{aligned}
\end{equation}

\begin{equation}
    \begin{aligned}\label{Unaccelerated infty}
        \lim_{\substack{k\rightarrow \infty \\ k\in S}} \nabla f(\bm{x}^{k})+L(\bm{x}^{k+1}-\bm{x}^{k})+ \lambda W^k \xi  = 0,
\end{aligned}
\end{equation}
where $\xi^k \in \partial \|\bm{x}^{k+1}\|_1$. \eqref{Accelerated infty} infers that $\bm{x}^*$ is first-order necessary optimal of \eqref{equation1} directly. Analogous to the proof in Lemma \ref{Lemma1} (ii), \eqref{Unaccelerated infty} can infer $\bm{x}^* = H_{\bm{x}}(\bm{x}^*,0)$, which further substantiates that $\bm{x}^*$ is first-order necessary optimal.
\end{proof}

\end{Theorem}

\section{Numerical results}

In this section, we assess the performance of different acceleration algorithms through numerical experiments. All code implementations are in MATLAB, and the experiments run on a 64-bit laptop equipped with MATLAB 2022a, an Intel Core™ i7-1165G7 processor (2.80GHz), and 32GB of RAM.

We test the following sparse recovery problem:

$$\min\limits _{\bm{x}\in \mathbb{R}^n} \:F(\bm{x}):=  \frac{1}{2}\|A\bm{x}-\bm{b}\|^2 +  \lambda \displaystyle\sum^n_{i=1} |x_i|^p,$$
where $A\in \mathbb{R}^{m\times n}, \bm{b}\in \mathbb{R}^{m}$, the target of the question is to recovery a sparse signal $\bm{x}$ from $m$ observations through the partial measurements with $m\ll n$. This is a common sparse optimization problem (similar to \cite{wen2018proximal,wen2018proximal,wang2022extrapolated}). We generate $m\times n$ matrix $A$ with i.i.d. standard Gaussian entries and orthonormalizing the rows. We generate the true signal $\bm{x}_{\mathop{{\rm true}}}$ by randomly selecting K elements from an n-dimensional all-zero vector and setting them to ±1. We form the observation vector $\bm{y} = A\bm{x}_{\mathop{{\rm true}}} + \bm{\epsilon}$, where $\bm{\epsilon}\in \mathbb{R}^n.$ follows a Gaussian distribution with mean 0 and variance $10^{-4}$. 

In the following experiments, we compare the performance of the Guard-AAIRL1, IRL1, Iteratively Reweighted $\ell_2$ algorithm (IRL2), and the state-of-the-art NesIRL1 algorithm (based on Nesterov acceleration) proposed in \cite{wang2022extrapolated}. We fix the regularization coefficient as $\lambda = 0.1$, $p=0.5$, $\bm{\epsilon}^0=\bm{1}$ and $\mu = 0.9$. We initialize all algorithms with $\bm{x}_0$ sampled from a standard Gaussian distribution. Furthermore, we apply a uniform termination criterion to all algorithms:$\mathop{{\rm opttol}} \le 10^{-14}$.

\begin{equation}
    \begin{aligned}
    \mathop{{\rm opttol}} = \frac{\|\bm{x}^k-\bm{x}^{k-1}\|}{\|\bm{x}^k\|},
\end{aligned}
\end{equation} 

\begin{table}
    \renewcommand\arraystretch{1.5}
    \caption{The performance of IRL1,IRL2,AAIRL1 and NesIRL1}\label{table2}
    \begin{tabular*}{\textwidth}{@{\extracolsep\fill}llllll}
    \toprule
    $(m,n,K)$         &  & cpu time &         &         &         \\ \cmidrule{3-6}
                    &  & IRL1     & IRL2    & AAIRL1  & NesIRL1 \\ \hline
    (400 , 800  , 80 )     &  & 2.0375   & 2.5625  & 1.3609  & 1.7094  \\ 
    (800 , 1600 , 160)     &  & 10.7656   & 12.8281 & 5.7422  & 8.0664  \\ 
    (1200, 2400 , 240)     &  & 28.2656  & 35.8125 & 17.4219 & 22.2851 \\ 
    (1600, 3200 , 320)     &  & 70.9688  & 78.6406 & 53.
    0859 & 63.2812 \\ \bottomrule

    (m,n,K)         &  & Sparsity &         &         &         \\ \cmidrule{3-6}
    &  & IRL1     & IRL2    & AAIRL1  & NesIRL1 \\ \hline
(400 , 800  , 80 )     &  & 68.4   & 386.7  & 67.1  & 61.2  \\ 
(800 , 1600 , 160)     &  & 148.5   & 713.6 & 143.9  & 132.1  \\ 
(1200, 2400 , 240)     &  & 207.1  & 1092.4 & 200.8 & 188.4 \\ 
(1600, 3200 , 320)     &  & 264.9  & 1422.7 & 255.4 & 241.2 \\ \bottomrule
    \end{tabular*}
    \end{table}

We use different scales of $(m,n,K)$ to verify the efficiency of these different algorithms. We measure the average CPU time across 50 random data sets $(A,\bm{x}_{\mathop{{\rm true}}},\bm{y})$. In the Guard-AAIRL1 algorithm, we employ commonly used settings $m=15,\eta = 0.85$ and $\beta = 10^{-11}$. Regarding NesIRL1, we choose $\alpha^k = (k-1)/ (k+2)$ and $\alpha^0 = 0$. Table \ref{table2} demonstrates that Guard-AAIRL1 exhibits the fastest convergence to the other algorithms. Additionally, the Fig. \ref{figure1} illustrates that the Guard-AAIRL1 algorithm requires significantly fewer iterations than the other algorithms.This is attributed to the larger computational cost per single step, possibly due to the subproblem [\ref{AA subproblem}] within the algorithm and the evaluation of nonmonotonic descent conditions [\ref{measure optimal}].

\begin{figure}
        \centering
        \centering
        \includegraphics[width=0.45\linewidth]{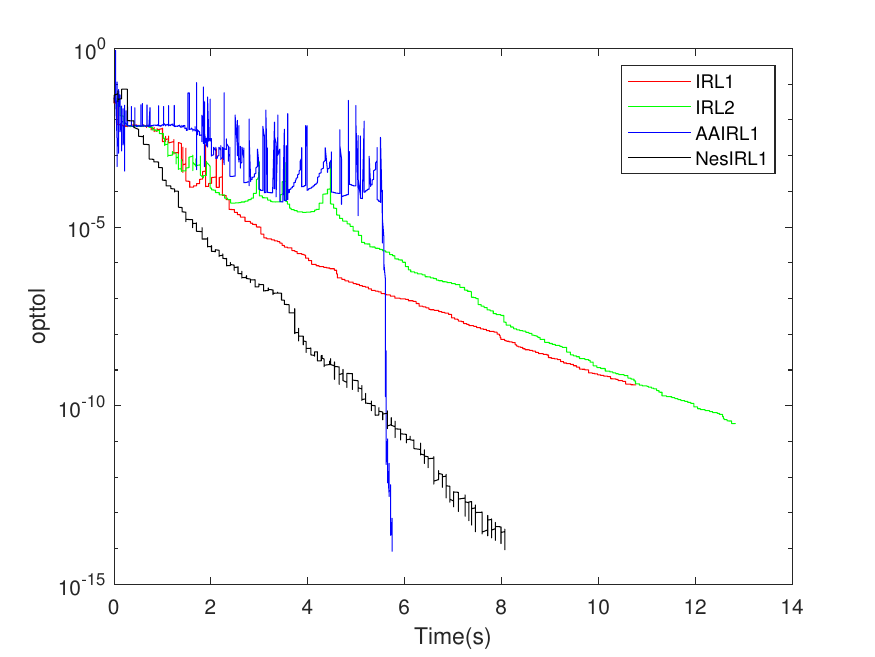}
        \centering
        \includegraphics[width=0.45\linewidth]{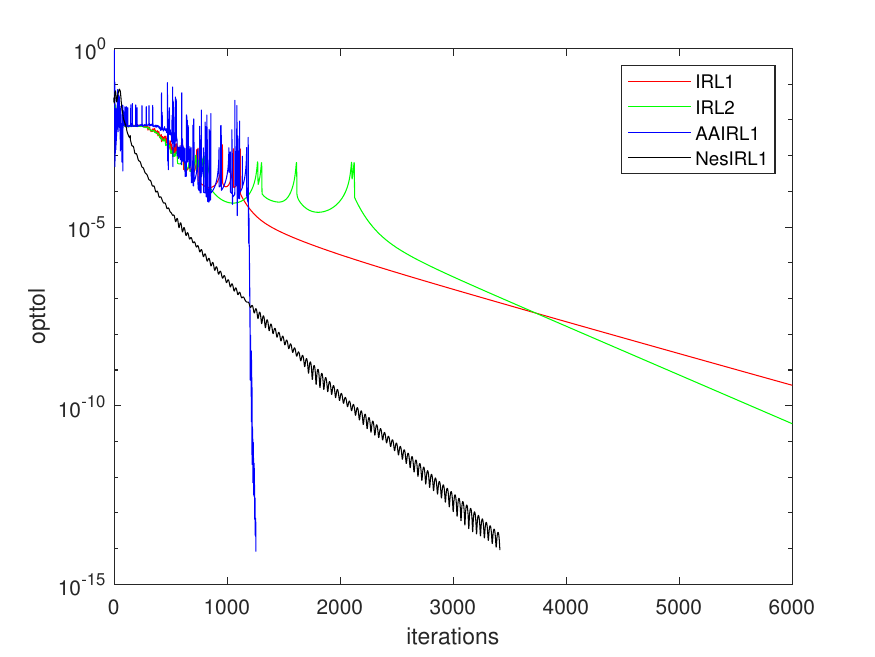}
        \caption{Opttol curves over the time and iterations when (m,n,K) = (800,1600,160)}\label{figure1}
\end{figure}


We now delve into the selection of various hyperparameters for the Guard-AAIRL1 algorithm. Fig. \ref{figure2} (a) illustrates that the choice of $m$ presents a trade-off . When $m$ is too small, Guard-AAIRL1 may not effectively leverage historical iteration information. Conversely, an excessively large $m$ results in an increased computational load per iteration. Empirical evidence suggests that reasonable choices for $m$ include $m=15$ or $m=5$. Furthermore, selecting appropriate values for $\eta$ and $\beta$ is crucial in controlling the degree of nonmonotonicity. In fact, the originator of this nonmonotonic line search condition recommended  $\eta=0.85$. Fig. \ref{figure2} (c) demonstrates a notable negative correlation between the value of $\beta$ and algorithm performance. At the tail of the algorithm, a larger value of $\beta$ will make the nonmonotonic descent condition is hard to be satisfied. Therefore, the value of $\beta$ show correspond to the torlerence of the problem.

\begin{figure}
    
    \subfigure[Performance for $m = 5,15,30$]{
    
        \includegraphics[width=0.45\linewidth]{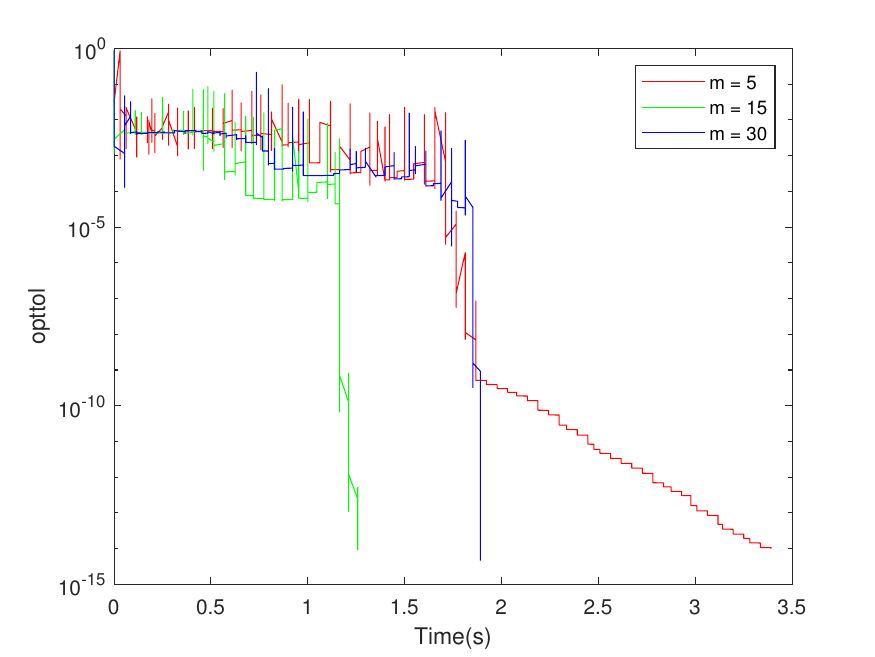}
        \label{label_for_cross_ref_1}

	  \includegraphics[width=0.45\linewidth]{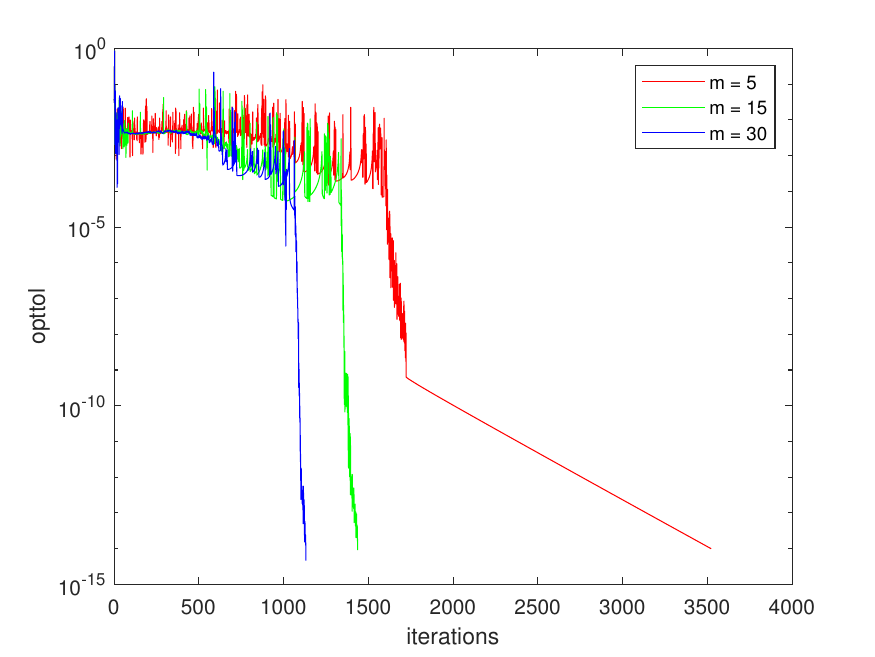}

    }
    
    \quad    
    
    \subfigure[Performance for $\eta = 0.2,0.35,0.5,0.7,0.85$]{
    	\includegraphics[width=0.45\linewidth]{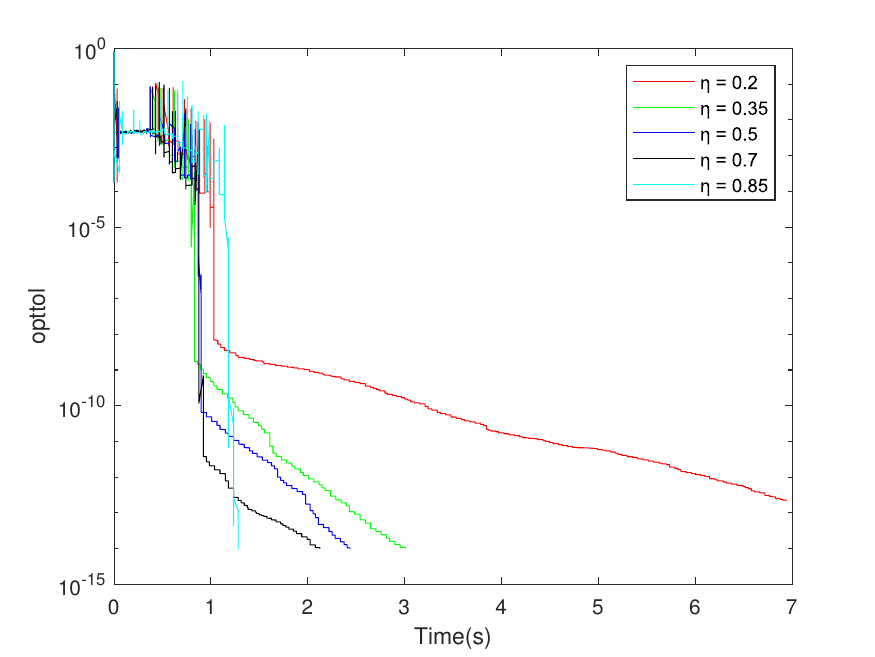}
        \includegraphics[width=0.45\linewidth]{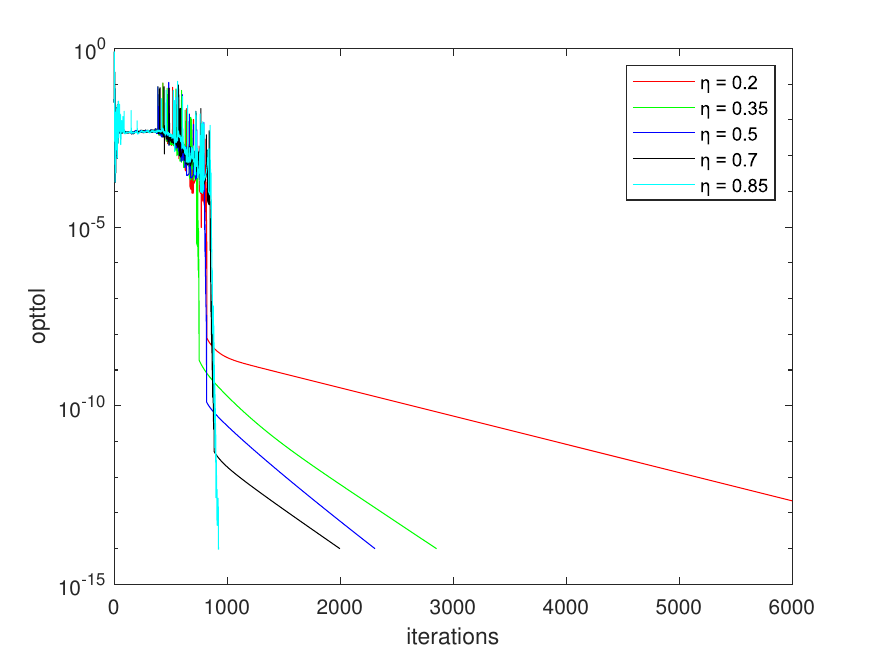}
        \label{label_for_cross_ref_4}
    }

    \quad    
    
    \subfigure[Performance for $\beta = 10^{-5},10^{-7},10^{-9},10^{-11},10^{-13}$]{
    	\includegraphics[width=0.45\linewidth]{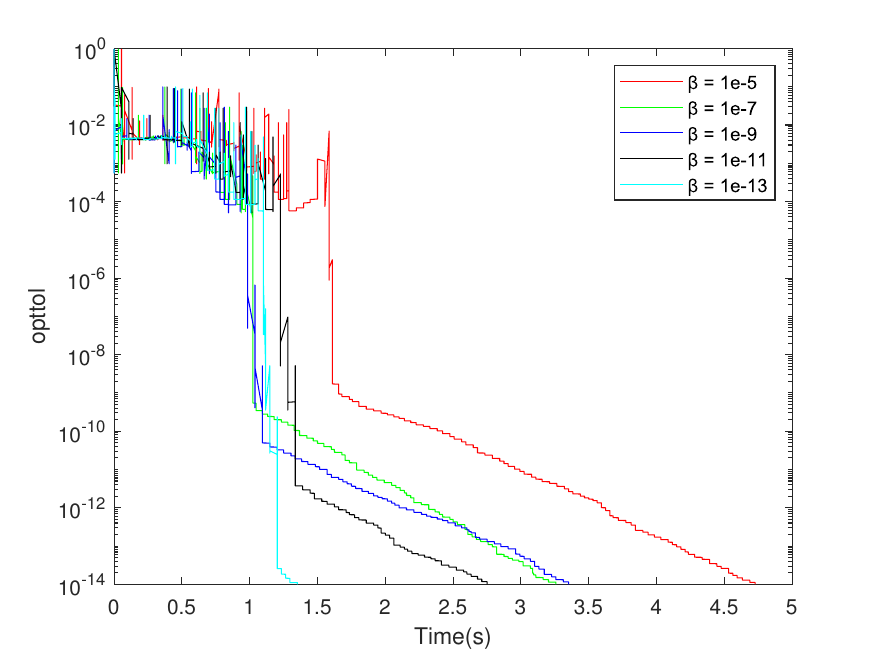}
     
	  \includegraphics[width=0.45\linewidth]{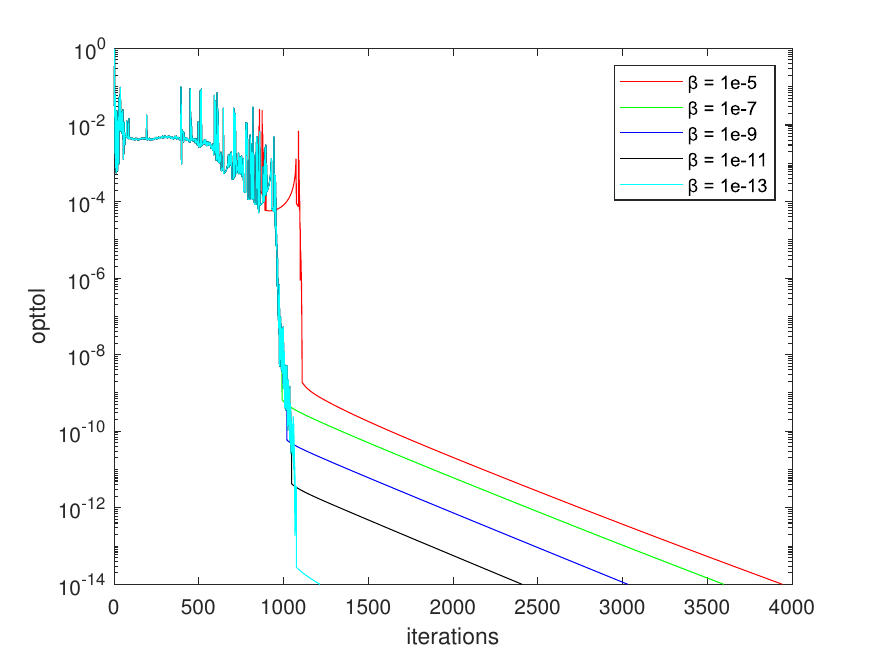}        
        \label{label_for_cross_ref_2}
    }
        \caption{The performance for different hyperparameter when (m,n,K) = (400,800,80)}.   \label{figure2}
\end{figure}

\newpage





\bibliography{sn_article}
\end{document}